\documentclass[a4paper]{amsart}
\usepackage{fullpage}
\usepackage[english]{babel}
\usepackage[utf8]{inputenc}
\usepackage{amsmath,amsfonts,amsthm}
\usepackage{mathtools}
\usepackage{graphicx}
\usepackage{subcaption}
\usepackage{type1cm}
\usepackage{eso-pic}
\usepackage{color}
\usepackage{pgfplotstable}
\usepackage{pgfplots}
\pgfplotsset{compat=1.5}
\usepackage{caption}
\makeatother
\usepackage{enumitem}
\usepackage{hyperref}
\usepackage{placeins}
\makeatletter
\AtBeginDocument{%
  \expandafter\renewcommand\expandafter\subsection\expandafter
    {\expandafter\@fb@secFB\subsection}%
  \newcommand\@fb@secFB{\FloatBarrier
    \gdef\@fb@afterHHook{\@fb@topbarrier \gdef\@fb@afterHHook{}}}%
  \g@addto@macro\@afterheading{\@fb@afterHHook}%
  \gdef\@fb@afterHHook{}%
}
\makeatother
\usepackage{tikz,tikz-cd}
\usepackage{float}
\usetikzlibrary{patterns}
\usetikzlibrary{calc}

\newtheorem{theorem}{Theorem}[section]
\newtheorem{lemma}[theorem]{Lemma}
\newtheorem{remark}[theorem]{Remark}
\newtheorem{problem}[theorem]{Problem}

\newtheorem{definition}[theorem]{Definition}
\newtheorem{proposition}[theorem]{Proposition}

\newcommand{\logLogSlopeTriangleInv}[6]
{
    \pgfplotsextra
    {
        \pgfkeysgetvalue{/pgfplots/xmin}{\xmin}
        \pgfkeysgetvalue{/pgfplots/xmax}{\xmax}
        \pgfkeysgetvalue{/pgfplots/ymin}{\ymin}
        \pgfkeysgetvalue{/pgfplots/ymax}{\ymax}

        \pgfmathsetmacro{\xArel}{#1}
        \pgfmathsetmacro{\yArel}{#3}
        \pgfmathsetmacro{\xBrel}{#1+#2}
        \pgfmathsetmacro{\yBrel}{\yArel}
        \pgfmathsetmacro{\xCrel}{\xArel}
        \pgfmathsetmacro{\lnxB}{\xmin*(1-(#1+#2))+\xmax*(#1+#2)} % in [xmin,xmax].
        \pgfmathsetmacro{\lnxA}{\xmin*(1-#1)+\xmax*#1} % in [xmin,xmax].
        \pgfmathsetmacro{\lnyA}{\ymin*(1-#3)+\ymax*#3} % in [ymin,ymax].
        \pgfmathsetmacro{\lnyC}{\lnyA+#4*(\lnxA-\lnxB)}
        \pgfmathsetmacro{\yCrel}{\lnyC-\ymin)/(\ymax-\ymin)} % THE IMPROVED EXPRESSION WITHOUT 'DIMENSION TOO LARGE' ERROR.

        \coordinate (A) at (rel axis cs:\xArel,\yArel);
        \coordinate (B) at (rel axis cs:\xBrel,\yBrel);
        \coordinate (C) at (rel axis cs:\xCrel,\yCrel);

        \draw[#5]   (A)-- node[pos=0.5,anchor=south] {1}
                    (B)-- 
                    (C)-- node[pos=0.5,anchor=west] {#6}
                    cycle;
    }
}

\title{The Johnson-N\'ed\'elec FEM-BEM coupling for magnetostatic problems in the isogeometric framework}

\author{Mehdi Elasmi}
\address{Technische Universität Darmstadt, Institute for Accelerator Science and Electromagnetic Fields, Schlossgartenstra\ss{}e 8, 64289 Darmstadt, Germany}
\email{elasmi@temf.tu-darmstadt.de}

\author{Christoph Erath}
\address{University College of Teacher Education Vorarlberg (PH Vorarlberg), Liechtensteinerstrasse 33-37, 6800 Feldkirch, Austria}
\email{christoph.erath@ph-vorarlberg.ac.at}

\author{Stefan Kurz}
\address{University of Jyväskylä, Faculty of Information Technology, PO Box 35, FI-40014 University of Jyväskylä, Finland}
\email{stefan.m.kurz@jyu.fi}

\thanks{M. Elasmi (corresponding author): Technische Universität Darmstadt, Germany; elasmi@temf.tu-darmstadt.de;
The research of this author was supported in parts by the \emph{Excellence Initiative} 
  of the German Federal and State Governments 
  and the \emph{Graduate School CE} within the \emph{Centre for Computational Engineering} at Technische Universit\"at Darmstadt.}
\thanks{C. Erath: University College of Teacher Education Vorarlberg (PH Vorarlberg), Austria; christoph.erath@ph-vorarlberg.ac.at}
\thanks{S. Kurz: University of Jyväskylä, Finland; stefan.m.kurz@jyu.fi}
\date{\textbf{\today}}
\DeclareMathOperator{\Div}{div}

\DeclareMathOperator{\curl}{curl}
\DeclareMathOperator{\bcurl}{{\bf curl}}
\DeclareMathOperator{\Span}{span}

\DeclareMathOperator{\I}{\mathcal I}
\DeclareMathOperator{\dtrace}{\boldsymbol{\gamma}_\mathrm{D}}
\DeclareMathOperator{\dtracex}{\boldsymbol{\gamma}_\mathrm{D}^\mathrm{e}}
\DeclareMathOperator{\ntrace}{\boldsymbol{\gamma}_\mathrm{N}}
\DeclareMathOperator{\ntraceM}{\boldsymbol{\gamma}^\M_\mathrm{N}}
\DeclareMathOperator{\ntracex}{\boldsymbol{\gamma}_\mathrm{N}^\mathrm{e}}
\DeclareMathOperator{\twtrace}{\boldsymbol{\gamma}_\times}

\newcommand{\Id}{{\operatorname{Id}}}
\newcommand{\opd}{{\,\operatorname{d}}}

\newcommand{\norm}[1]{{\left\lVert #1 \right\rVert}}
\newcommand{\eqnorm}[1]{{\left\vert\kern-0.25ex\left\vert\kern-0.25ex\left\vert #1 
    \right\vert\kern-0.25ex\right\vert\kern-0.25ex\right\vert}}
    
\newcommand{\bh}{{H^1(\Omega)}}
\newcommand{\bH}{{\boldsymbol{\mathcal{H}}}}
\newcommand{\bHl}{{\boldsymbol{\mathcal{H}}_\ell}}

\newcommand{\hcurl}{{\boldsymbol{H}(\bcurl,\Omega)}}

\newcommand{\hzgcurl}{{\boldsymbol{H}_{\boldsymbol{0}}(\bcurl,\Omega)}}
\newcommand{\hdiv}{{\boldsymbol{H}(\Div,\Omega)}}
\newcommand{\hdivK}{{\boldsymbol{H}(\Div,K)}}

\newcommand{\hdivzh}{{\boldsymbol{H}_0(\Div 0,\Omega)}}

\newcommand{\hcurlcurl}{{\boldsymbol{H}(\bcurl\,\bcurl,\Omega)}}
\newcommand{\hcurlcurlK}{{\boldsymbol{H}(\bcurl\,\bcurl,K)}}
\newcommand{\hcurlcurle}{{\boldsymbol{H}_\mathrm{loc}(\bcurl\,\bcurl,\Omega^\mathrm{e})}}

\newcommand{\hcurlK}{{\boldsymbol{H}(\bcurl,K)}}

\newcommand{\hcurle}{{\boldsymbol{H}_\mathrm{loc}(\bcurl,\Omega^\mathrm{e})}}

\newcommand{\bV}{{\boldsymbol{V}_\ell}}
\newcommand{\bVz}{{\boldsymbol{V}_{0,\ell}}}
\newcommand{\bVzp}{{\boldsymbol{V}_{0,\ell}^\perp}}

\newcommand{\bX}{{\boldsymbol{X}_{\ell}}}
\newcommand{\bXz}{{\boldsymbol{X}_{0,\ell}}}

\newcommand{\hcurlg}{{\boldsymbol{H}^{-\frac{1}{2}}(\curl_\Gamma,\Gamma)}}

\newcommand{\hdivg}{{\boldsymbol{H}^{-\frac{1}{2}}(\Div_\Gamma,\Gamma)}}

\newcommand{\hdivzg}{{\boldsymbol{H}^{-\frac{1}{2}}(\Div_\Gamma 0,\Gamma)}}

\newcommand{\bl}{{\boldsymbol{L}^2(\Omega)}}
\newcommand{\ble}{{\boldsymbol{L}_\mathrm{loc}^2(\Omega^\mathrm{e})}}
\newcommand{\spht}{{H^{\frac{1}{2}}(\Gamma)}}
\newcommand{\dht}{{H^{-\frac{1}{2}}(\Gamma)}}

\newcommand{\hcurlnormq}[1]{{\left\lVert #1 \right\rVert^2_{\operatorname{\hcurl}}}}

\newcommand{\lnormq}[1]{{\left\lVert #1 \right\rVert^2_{\operatorname{\bl}}}}
\newcommand{\abs}[1]{{\left\lvert #1 \right\rvert}}
\newcommand{\duality}[1]{{\left\langle #1 \right\rangle}}
\newcommand{\scalar}[1]{{\left( #1 \right)_\Omega}}

\newcommand{\scalarcurl}[1]{{\left( #1 \right)_{\bcurl,\Omega}}}

\newcommand{\R}{{\mathbb R}}
\newcommand{\bn}{{\boldsymbol n}}
\newcommand{\bx}{{\boldsymbol x}}
\newcommand{\by}{{\boldsymbol y}}

\newcommand{\ba}{{\boldsymbol b}}

\newcommand{\bell}{{\boldsymbol \ell}}

\newcommand{\bu}{{\boldsymbol u}}
\newcommand{\bul}{{\boldsymbol u}_\ell}
\newcommand{\uvec}{{\mathbf{u}}}
\newcommand{\uvecl}{{\mathbf{u}_\ell}}
\newcommand{\bv}{{\boldsymbol v}}
\newcommand{\bvl}{{\boldsymbol v}_\ell}
\newcommand{\vvec}{{\mathbf{v}}}
\newcommand{\vvecl}{{\mathbf{v}_\ell}}
\newcommand{\bw}{{\boldsymbol w}}

\newcommand{\wvec}{{\mathbf{w}}}

\newcommand{\bpsi}{{\boldsymbol \psi}}
\newcommand{\bpsil}{{\boldsymbol \psi}_\ell}

\newcommand{\bxi}{{\boldsymbol \xi}}

\newcommand{\bvphi}{{\boldsymbol \varphi}}
\newcommand{\bXi}{{\boldsymbol \Xi}}
\newcommand{\bXit}{\widetilde{\boldsymbol \Xi}}
\newcommand{\bi}{{\boldsymbol i}}
\newcommand{\bj}{{\boldsymbol j}}
\newcommand{\bc}{{\boldsymbol c}}
\newcommand{\bp}{{\boldsymbol p}}
\newcommand{\bpt}{\widetilde{\boldsymbol p}}
\newcommand{\bk}{{\boldsymbol k}}
\newcommand{\rom}[1]{\uppercase\expandafter{\romannumeral #1\relax}}
\newcommand{\bphi}{{\boldsymbol \phi}}
\newcommand{\bphil}{{\boldsymbol \phi}_\ell}
\newcommand{\bue}{{\boldsymbol{u}^\mathrm{e}}}
\newcommand{\bve}{{\boldsymbol{v}^\mathrm{e}}}

\newcommand{\V}{{\mathcal V}}
\newcommand{\Cal}{{\mathcal P}}
\newcommand{\SL}{{\mathcal A_0}}
\newcommand{\Astab}{{\widetilde{\mathcal{B}}}}

\newcommand{\DL}{{\mathcal C_0}}
\newcommand{\Steklov}{{\mathcal S}^\mathrm{int}}
\newcommand{\Steklove}{{\mathcal S}^\mathrm{ext}}
\newcommand{\DLA}{{\mathcal B_0}}
\newcommand{\M}{{\mathcal U}}
\newcommand{\HS}{{\mathcal N_0}}
\newcommand{\bPsi}{{\boldsymbol \Psi}}

\begin{document}

\begin{abstract}
We consider a Johnson-N\'ed\'elec FEM-BEM coupling, which is a direct and non-symmetric coupling of finite and boundary element methods, in order to solve interface problems for the magnetostatic Maxwell's equations with the magnetic vector potential ansatz. In the FEM-domain, equations may be non-linear, whereas they are exclusively linear in the BEM-part to guarantee the existence of a fundamental solution. First, the weak problem is formulated in quotient spaces to avoid resolving to a saddle point problem. Second, we establish in this setting well-posedness of the arising problem using the framework of Lipschitz and strongly monotone operators as well as a stability result for a special type of non-linearity, which is typically considered in magnetostatic applications. Then, the discretization is performed in the isogeometric context, i.e., the same type of basis functions that are used for geometry design are considered as ansatz functions for the discrete setting. In particular, NURBS are employed for geometry considerations, and B-Splines, which can be understood as a special type of NURBS, for analysis purposes. In this context, we derive a~priori estimates w.r.t. $h$-refinement, and point out to an interesting behavior of BEM, which consists in an amelioration of the convergence rates, when a functional of the solution is evaluated in the exterior BEM-domain. This improvement may lead to a doubling of the convergence rate under certain assumptions. Finally, we end the paper with a numerical example to illustrate the theoretical results, along with a conclusion and an outlook.\\
 \bigskip

\noindent \textbf{Keywords.} Finite Element Method, Boundary Element Methods, Johnson-N\'ed\'elec coupling, non-symmetric coupling, Isogeometric Analysis, non-linear operators, curl-curl equations, quotient spaces,  magnetostatics, vector potential formulation, well-posedness, a~priori estimates, super-convergence \bigskip \\
\noindent \textbf{Mathematics Subject Classification.}
\end{abstract}
\maketitle
\section{Introduction} \label{sec:introduction}
In modern engineering sciences, a broad spectrum of applications such as antennas, electric machines, linear accelerators, etc., involves electromagnetic fields. Mathematically, the foundation of electromagnetism in its classical form relies on a set of four partial differential equations called Maxwell's equations, see, e.g., \cite{griffiths2005introduction,jackson2009classical} for an introduction to the subject. In particular, when the currents and charges are time-invariant, the electric and magnetic components of the field can be considered separately, which leads to the so called electrostatic and magnetostatic cases, respectively. In this work, we will focus on the latter. For more information about the formulations and equations encountered in magnetostatics, we refer the reader to \cite{touzani2014mathematical} and the references cited therein, for instance. We consider the following model problem.
\begin{problem}\label{prob:strongForm} Let $\Omega \subset \R^3$ be a Lipschitz bounded domain with boundary $\Gamma$ and exterior domain $\Omega^\mathrm{e}:= \R^3\backslash\overline{\Omega}$. We aim to find $(\bu,\bue)$ for some appropriate right-hand sides $(\boldsymbol{f},\boldsymbol{u}_0,\boldsymbol{\phi}_0)$ such that
	\begin{subequations}
		\begin{align}
		\bcurl \M \bcurl \bu &= \boldsymbol{f} && \text{in } \Omega, \label{eq:curlcurlInt}\\
		\bcurl \bcurl \bue &= \boldsymbol{0} && \text{in } \Omega^\mathrm{e}, \label{eq:curlcurlExt}\\
		\Div \bu & = 0 && \text{in } \Omega,\label{eq:divFreeInt} \\ 
		\Div \bue & = 0 && \text{in } \Omega^\mathrm{e},\label{eq:divFreeExt} \\
		\dtrace \bu - \dtracex \bue &= \boldsymbol{u}_0 && \text{on } \Gamma, \label{eq:DirichletJump}\\
		\ntraceM \bu - \ntracex \bue &= \boldsymbol{\phi}_0 && \text{on } \Gamma,\label{eq:NeumannJump}\\
		\bue &= \mathcal{O}\left(\abs{\bx^{-1}}\right) && \text{for } \abs{\bx}\rightarrow \infty. \label{eq:radiationCond}
		\end{align}
	\end{subequations}
	%Thereby, $\M : \R^3 \rightarrow \R^3$ is a possibly non-linear operator, and $\dtrace$ and $\ntrace$ denote the Dirichlet and Neumann trace operators, respectively. The trace operators $\dtrace$ and $\ntrace$ will be characterized below.
\end{problem}  
\noindent In the context of magnetostatics, \eqref{eq:curlcurlInt} and \eqref{eq:curlcurlExt} model the magnetic vector potential inside and outside the domain of interest $\Omega$, respectively. Thereby, the right-hand side $\boldsymbol{f}$ is a compactly supported current density in $\Omega$, and $\M$ is a possibly non-linear reluctivity tensor, which models the material's magnetic properties of the domain. The absence of a material tensor in the second equation means that we are dealing with exterior domains that are exclusively filled with linear materials, where the simplest form is air. Equations \eqref{eq:divFreeInt} and \eqref{eq:divFreeExt} are imposed gauging conditions that guarantee the uniqueness of the vector fields, cf. \cite[Section~2.2.4]{touzani2014mathematical}. Since we are dealing with an interface problem, the transition of the field from $\Omega$ to $\Omega^\mathrm{e}$ is not necessarily continuous. For this, and in order to keep the formulation general, we prescribe in \eqref{eq:DirichletJump} and \eqref{eq:NeumannJump} boundary conditions for both the Dirichlet and Neumann data. This can be interpreted as a jump condition of the tangential component of the magnetic flux, and as a jump of the normal component of the magnetic field across the boundary $\Gamma$, respectively. The trace operators $\dtrace$ and $\ntraceM$ (and $\dtracex$ and $\ntracex$) will be characterized below. The last equation \eqref{eq:radiationCond} determines the decay behavior of the field at infinity. This decay condition is needed to guarantee uniqueness of the exterior solution. \\

We aim to solve Problem \ref{prob:strongForm} numerically. A straightforward choice is a coupling of Finite Element Method (FEM) and Boundary Element Methods (BEM). This allows, by using FEM in bounded domains, an efficient treatment of non-linear materials and non-homogeneous equations, and by using BEM for unbounded domains, an attractive way to deal with exterior problems. Note that BEM is only applicable, if a fundamental solution exists. In this case, this condition is fulfilled, since we are considering linear materials in the exterior domain. This leads to the same type of fundamental solution as for the Laplace equation. In contrast to FEM only approaches in unbounded domains, BEM avoids the meshing of a fictitious domain surrounding the region of interest, and the necessary truncation of the whole domain. Hence, it avoids the imposition of non-physical boundary conditions that yields approximation errors. In fact, BEM relies on an integral formulation that transfers the problem from the exterior domain to the boundary. Thus, the task reduces to find the complete Cauchy data on the boundary. A representation formula leads to the solution in every point of the domain in dependence of the Cauchy data and fundamental solution. Moreover, as demonstrated  in \cite{EEK20} for the two-dimensional Laplacian interface problem, the possible super-convergence of BEM may be particularly advantageous, especially if we are aiming for the solution or some derived quantities evaluated in the exterior domain. For example, the calculation of forces and torques may profit from this behavior of BEM, which concretely may double the convergence rate in certain circumstances.\\
In the literature, several manuscripts considered FEM and/or BEM to solve magnetostatic problems. In the following, we only mention a few of them for convenience. For example, we refer the reader to \cite{coulomb1981finite,bossavit1988rationale,magele1988comparison}, for FEM only, and to \cite{lindholm1980notes,mayergoyz1987new,buchau2003comparison} for BEM only approaches. For the FEM-BEM coupling, we could classify the methods depending on the considered unknowns. The first type of approaches defines auxiliary quantities. For instance, \cite{frangi2003magneto,hafla2006accuracy,salgado2008symmetric} are based on the introduction of a total and/or reduced scalar magnetic potential. These methods suffer from a lack of stability for high permeable materials. In \cite{kuhn2002symmetric}, this drawback was circumvented for simply-connected domains by employing a vector potential formulation for the FEM part, and by the introduction of a total scalar magnetic potential in the exterior BEM domain. Afterwards, the consideration of a reduced scalar magnetic potential for the BEM part, as proposed in \cite{pusch2010robust}, relaxed the topology assumption of the previous scheme. The second type of approaches would be to consider Problem \ref{prob:strongForm} ``as it is'', i.e., a vector potential formulation for both the FEM and BEM domains. This is the case in \cite{kielhornrobust} for the eddy-current problem. Note that the magnetostatic equations appear as a limiting case of the eddy-current approximation, namely, in the absence of conducting regions. In this context, we recall some other works that adopt a similar approach: a mathematical foundation of the symmetric BEM-BEM coupling for eddy-currents is presented in \cite{hiptmair2005coupled}, and a rigorous analysis of the symmetric FEM-BEM coupling for an $E$-based eddy-current model can be found in \cite{hiptmairSymmetric}.   
Another classification factor could be the type of coupling, i.e., the equations of BEM that are used to supplement the FEM part. We distinguish three main coupling types: a direct symmetric coupling \`a la Costabel \cite{costabel1988symmetric}, an indirect non-symmetric coupling as introduced by Bielak and MacCamy in \cite{bielak1983exterior}, and a Johnson-N\'ed\'elec coupling, which is a direct non-symmetric scheme, see \cite{johnson1980coupling}. In the cited references, the FEM-BEM couplings are mostly based on the symmetric type. Its validity and well-posedness are demonstrated both analytically and numerically for several Maxwell based models. Beside the cited literature, see also \cite{hiptmair2003coupling} for the electromagnetic scattering problem. In \cite{kielhornrobust}, and \cite{BRUCKNER20121862}, a numerical validation is also provided for the second type of coupling dedicated to eddy-current problems using a vector potential formulation, and for the Johnson-N\'ed\'elec coupling with a reduced magnetic potential approach in magnetostatics, respectively. In this work, we address the mathematical foundation of the direct non-symmetric Galerkin FEM-BEM coupling for the vector potential based formulation. To the authors' knowledge, this is done for the presented type of equations for the first time in this paper. In addition, we discretize the obtained variational formulation using an isogeometric framework. Isogeometric Analysis (IGA) was first proposed in \cite{HUGHES20054135}. It relies on B-Splines and Non-Uniform Rational B-Splines (NURBS), which are widely used in Computer Aided Design (CAD) softwares, see \cite{piegl2012nurbs}, for an extensive introduction to B-Splines and NURBS. Hence, employing the same type of basis functions for the design and the numerical discretization facilitates the unification of two processes that are typically separated, namely, design and analysis. This allows an exact representation of the geometry and an efficient $h$-refinement that does not alter the original geometry. Moreover, $p$-refinement is merely achieved thanks to the definition of B-Splines and NURBS. IGA reveals to be a promising alternative to classical FEM and BEM in many electromagnetic applications, see \cite{corno2016isogeometric,dolz2019numerical}, for instance. A rigorous mathematical analysis for isogeometric FEM started in \cite{Bazilevs2006}, and for isogeometric Galerkin BEM in \cite{GantnerGregor2014AiB}. For more details and references we refer to \cite{EEK20}.\\
The rest of this paper is organized as follows: Section \ref{sec:preliminaries} is devoted to the introduction of the involved mathematical framework, i.e., the relevant vector spaces,  trace operators, and surface differential operators. The definitions given herein are based on the results of the pioneering works \cite{buffaTracesIntegrationParts,buffa2001traces,buffaTraces}. Then, we dedicate Section \ref{sec:weakFormulation} to the derivation of variational formulations for the interior and exterior problems. We adopt a formulation of the problem in quotient spaces, as done in \cite{KurzDF}, to separate the infinite-dimensional kernel of the $\bcurl$ operator, which avoids the introduction of a saddle point problem. With this, we derive in Section \ref{sec:fem-bem} a direct non-symmetric FEM-BEM coupling, and establish well-posedness and stability of the formulation using the framework of Lipschitz continuous and strongly monotone operators, such as given in \cite{zeidler}. In Section \ref{sec:iga}, we introduce the conforming isogeometric discretization scheme, analyze the resulting discrete setting, and derive a priori estimates for $h$-refinements. For this, we rely particularly on the results of \cite{buffa2010isogeometric} and \cite{buffa2020multipatch}, which provide approximation properties of B-Splines in $\hcurl$ and its trace spaces, respectively. The penultimate Section \ref{sec:results} is dedicated to the numerical confirmation of the theoretical results, and we complete the work with a conclusion and an outlook in Section \ref{sec:conclusion}. The focus of this work is on mathematical modeling and numerical analysis. The numerical example serves as a simple but representative proof-of-concept to validate the theoretical convergence rates. More realistic engineering examples shall be deferred to future work.        
\section{Functional setting} \label{sec:preliminaries}
We present in this section the mathematical framework that is needed for the derivation of a variational formulation for Problem \ref{prob:strongForm}.
Throughout this work, we consider $\Omega \subset \R^3$ to be a Lipschitz bounded domain with boundary $\Gamma$. Moreover, $\Omega$ is assumed to be simply-connected with connected boundary $\Gamma:= \partial \Omega$. Furthermore, we denote by $\Omega^\mathrm{e}$ the exterior domain w.r.t. $\Omega$, i.e., $\Omega^\mathrm{e} = \R^3\backslash\overline{\Omega}$, and by $\bn(\bx)$ an outer normal vector w.r.t. $\Omega$ that is evaluated at $\bx \in \Gamma$. \\

First, we start with the corresponding energy spaces in the interior and exterior domains. 
 \begin{align*}
  \hcurl &:= \{ \bv \in \bl: \, \bcurl \bv \in \bl \},\\
  \hcurle &:= \{ \bve \in \ble: \, \bcurl \bve \in \ble \},\\
  \hcurlcurl &:= \{ \bv \in \hcurl: \, \bcurl\bcurl \bv \in \bl \},\\
  \hcurlcurle &:= \{ \bve \in \hcurle: \, \bcurl\bcurl \bve \in \ble \},  
 \end{align*}
where $\bl$ is a standard Lebesgue space, which has to be understood here in a vectorial sense, i.e., $\bl := \left( L^2(\Omega) \right)^d $, $d\leq 3$, and $\ble$ denotes the space of functions with local $\boldsymbol{L}^2$-behavior, i.e.,
\begin{equation*}
\ble := \{ \bve : \Omega^\mathrm{e} \rightarrow \R \vert \, \bve_{\vert K} \in \boldsymbol{L}^2(K) \text{ for all } K \subset \overline{\Omega^\mathrm{e}} \text{ compact} \}.
\end{equation*}
The energy spaces are equipped with the usual inner products, which induce natural graph norms. For instance, the $\hcurl$ inner product and its induced norm read 
\begin{align*}
\scalarcurl{\bv,\bw} &:= \scalar{\bv,\bw} + \scalar{\bcurl \bv,\bcurl \bw}, && \bv,\bw \in \hcurl,\\
\hcurlnormq{\bv} &:=  \scalarcurl{\bv,\bv} = \norm{\bv}_\bl^2 + \norm{\bcurl\bv}_\bl^2, && \bv \in \hcurl,
\end{align*}	
respectively.
% \begin{align*}
%  \scalarcurl{\bu,\bv} &:= \scalar{\bu,\bv} + \scalar{\bcurl \bu,\bcurl \bv}, && \bu,\bv \in \hcurl,\\
%  \scalarcurle{\bue,\bve} &:= \scalare{\bue,\bve} + \scalare{\bcurl \bue,\bcurl \bve}, && \bue,\bve \in \hcurle,\\
%  \scalarcurlcurl{\bu,\bv} &:= \scalar{\bu,\bv} + \scalar{\bcurl\bcurl \bu,\bcurl\bcurl \bv}, && \bu,\bv \in \hcurlcurl,\\
%  \scalarcurlcurle{\bue,\bve} &:= \scalare{\bue,\bve} + \scalare{\bcurl\bcurl \bue,\bcurl\bcurl \bve}, && \bue,\bve \in \hcurlcurle,
% \end{align*}
% and the graph norms
% \begin{align*}
% \hcurlnormq{\bu} &:=  \scalarcurl{\bu,\bu}, && \bu \in \hcurl\\
% \hcurlnormeq{\bue} &:=  \scalarcurle{\bue,\bue}, && \bue \in \hcurle,\\
% \hcurlcurlnormq{\bu} &:=  \scalarcurlcurl{\bu,\bu}, && \bu \in \hcurlcurl\\
% \hcurlcurlnormeq{\bue} &:=  \scalarcurlcurle{\bue,\bue}, && \bue \in \hcurlcurle.
% \end{align*}
Then, we require suitable trace spaces for $\hcurl$. For this, surface differential operators that act on tangential vector fields play an important role. Let $\nabla_\Gamma$ and $\bcurl_\Gamma$ denote the surface gradient and the vectorial surface curl-operator, respectively. In particular, we will encounter in this work their adjoint operators $\Div_\Gamma$ and $\curl_\Gamma$. They are defined by duality as
\begin{equation}\label{eq:diffOperator}
\duality{\Div_\Gamma \bv_{|\Gamma},w_{|\Gamma}}_\Gamma = -\duality{\bv_{|\Gamma},\nabla_\Gamma w_{|\Gamma}}_\Gamma, \quad \duality{\curl_\Gamma \bv_{|\Gamma},w_{|\Gamma}}_\Gamma = \duality{\bv_{|\Gamma},\bcurl_\Gamma w_{|\Gamma}}_\Gamma,
\end{equation}
and are called the surface divergence operator and the scalar surface curl-operator, respectively. Thereby, $\duality{\cdot,\cdot}_\Gamma$ denotes the natural duality pairing, which is obtained by the extended $\boldsymbol{L}^2$-scalar product. A rigorous definition of the above mentioned surface differential operators for Lipschitz domains, and the relation \eqref{eq:diffOperator} are found in \cite[Section~3]{buffaTraces}. In the same work, concretely in \cite[Section~4]{buffaTraces}, a characterization of suitable trace spaces for $\hcurl$ is given. In the following, we present some relevant results for convenience, and refer the reader to the cited paper, and e.g., to \cite{hiptmairSymmetric} for more details.
\begin{definition}
Let $\bv \in \boldsymbol{C}_0^\infty(\overline{\Omega})$ be a smooth vector field with compact support in $\overline{\Omega}$. The tangential (Dirichlet) surface trace $\dtrace$ and the twisted tangential surface trace operator $\twtrace$ read 
\begin{align*}
\dtrace \bv (\bx) &:=  \bn(\bx) \times \left( \bv(\bx){\vert_\Gamma} \times \bn(\bx) \right), \\
 \twtrace \bv (\bx) &:=  \bv(\bx){\vert_\Gamma} \times \bn(\bx)
\end{align*}
for a.e. $\bx \in \Gamma$. By extending $\dtrace$ and $\twtrace$ by continuity to $\boldsymbol{H}^1(\Omega)$, we define $V_\times:=\twtrace(\boldsymbol{H}^1(\Omega))$ and $V_\mathrm{D}=\dtrace(\boldsymbol{H}^1(\Omega))$. With this, the relevant trace spaces are given by
\begin{align*}
\hdivg &:= \{\bpsi \in V_\times^\prime: \Div_\Gamma \bpsi \in \dht\} ,\\
\hcurlg &:= \{\bpsi \in V_\mathrm{D}^\prime: \curl_\Gamma \bpsi \in \dht\},
\end{align*}
where $V_\times^\prime$ and $V_\mathrm{D}^\prime$ denote the dual spaces of $V_\times$ and $V_\mathrm{D}$, respectively. The trace spaces $\hdivg$ and $\hcurlg$ are endowed with the natural graph norms, where the $V_\times$ and $V_\mathrm{D}$ norms are given by
\begin{align*}
\norm{\bpsi}_{V_\times} &:= \inf_{\bxi \in \boldsymbol{H}^\frac{1}{2}(\Gamma)} \{ \norm{\bxi}_{\boldsymbol{H}^\frac{1}{2}(\Gamma)}: \twtrace \bxi = \bpsi\},\\
\norm{\bpsi}_{V_\mathrm{D}} &:= \inf_{\bxi \in \boldsymbol{H}^\frac{1}{2}(\Gamma)} \{ \norm{\bxi}_{\boldsymbol{H}^\frac{1}{2}(\Gamma)}: \dtrace \bxi = \bpsi\},
\end{align*}
respectively. Thereby, $\boldsymbol{H}^\frac{1}{2}(\Gamma)$ is the standard trace space of $\boldsymbol{H}^1(\Omega)$.
\end{definition}
 \begin{remark} Along an edge, a function in $V_\mathrm{D}$ (in $V_\times$) exhibits a weak tangential continuity (a weak normal continuity).
\end{remark}%
\noindent From \cite[Theorem~4.1]{buffaTraces}, we also know that the following extensions are valid, such that   
\begin{align*}
\dtrace: \hcurl &\rightarrow \hcurlg, \\
\twtrace: \hcurl & \rightarrow \hdivg
\end{align*}
are linear, continuous, and surjective. Moreover, note that the spaces $\hcurlg$ and $\hdivg$ are dual, see \cite[Section~4]{buffa2001traces}. This duality can be seen from the Green's formula,
 \begin{equation}\label{eq:Green1}
\scalar{\bw,\bcurl \bv } - \scalar{\bcurl \bw,\bv}  = \duality{\twtrace \bw, \dtrace \bv}_\Gamma,
\end{equation}
which holds for all $\bv,\bw \in \hcurl$, see \cite{buffaTracesIntegrationParts}, for instance. \\
To complete the Cauchy data, let $\ntraceM \bv := \twtrace \M \bcurl\bv$ denote the Neumann trace operator. Thereby, and throughout the rest of this work, $\M:\R^3 \rightarrow \R^3$ is a possibly non-linear operator, which is assumed to be Lipschitz continuous and strongly monotone, i.e., 
\begin{enumerate}[label=(A\arabic*)]
	\item \label{A1} Lipschitz continuity: 
	\[ \exists \, C^\M_\mathrm{L}>0  \text{ such that } \vert \mathcal{U} \bx - \mathcal{U}\by \vert \leq C^\M_\mathrm{L}\, \vert \bx - \by\vert \qquad \forall \bx,\by \in \R^3.\]
	\item \label{A2} strong monotonicity: \[ \exists \, C^\M_\mathrm{M}>0 \text{ such that }  
	\scalar{ \M \bv - \M \bw, \bv - \bw } \geq C^\M_\mathrm{M}\, \norm{\bv - \bw}_\bl^2 \qquad \forall \bv,\bw \in \bl  .\]
\end{enumerate}  
By integrating the square of \ref{A1}, we obtain for $\bv,\bw \in \hcurl$ 
\begin{equation*}
\lnormq{ \mathcal{U} \bcurl \bv - \mathcal{U}\bcurl\bw } \leq (C^\M_\mathrm{L})^2 \, \lnormq{ \bcurl\bv - \bcurl\bw}.
\end{equation*}
Hence, $\ntraceM \bv$ is well-defined. Moreover, inserting $\bw := \M\bcurl \bu$ in \eqref{eq:Green1} yields for $\bu \in \hcurlcurl$
\begin{equation}\label{eq:Green2}
\scalar{\M\bcurl\bu,\bcurl \bv }-\scalar{\bcurl\M\bcurl \bu,\bv}  = \duality{\ntraceM \bu, \dtrace \bv}_\Gamma .
\end{equation}
Therefore, it follows that \[\ntraceM : \hcurlcurl \rightarrow \hdivg\] is also linear, continuous, and surjective \cite[Lemma~3.3]{hiptmairSymmetric}. For notational simplicity, whenever $\M = \Id$, with $\Id$ denoting the identity operator, we write $\ntrace$ instead of $\ntraceM$. \\

\begin{figure}[t!]
\begin{tikzcd}
H^1(\Omega) \arrow[r, "\nabla"] \ar[dd,"\gamma_0"]
& \hcurl \arrow[r, "\bcurl"] \arrow[d,"\dtrace"] \ar[ddd, "\twtrace"pos=0.75, bend right=55]
& \hdiv \arrow[r, "\Div"] \ar[dd,"\gamma_\bn"]
& L^2(\Omega)\\
& \hcurlg \ar[dr,"\curl_\Gamma"] \ar[dd,"\times\bn"]
&
&[1.5em] \\
\spht \ar[ur, "\nabla_\Gamma"] \ar[dr, "\bcurl_\Gamma"']
&
& \dht \\
& \hdivg \ar[ur,"\Div_\Gamma"']
&
&
\end{tikzcd} 
\caption{The de~Rham complex in a three-dimensional domain $\Omega$ and on its boundary $\Gamma$.}\label{fig:deRham}
 \end{figure}%
For convenience, we summarize in Figure \ref{fig:deRham} the spaces, the trace operators, and differential operators in a de~Rham complex, see \cite[Fig.~2]{buffa2020multipatch}. Thereby, the trace $\gamma_\bn: \hdiv \rightarrow \dht$ denotes the standard normal trace, applied to vector fields, and $\gamma_0: H^1(\Omega) \rightarrow \spht$ is the standard trace operator in the scalar case, see \cite[Chapter~1]{steinbach}, for instance. \\

For the exterior domain, we denote the exterior traces by $\gamma^\mathrm{e}_\bn$, $\dtracex$ and $\ntracex$.  They are defined analogously to $\gamma_\bn$, $\dtrace$ and $\ntrace$ with respect to the exterior domain $\Omega^\mathrm{e}$, respectively, and the same mapping properties hold. \\ 
 
 \section{Weak formulations and quotient spaces} \label{sec:weakFormulation}
 This section is devoted to the derivation of boundary integral equations for the exterior problem, and of a weak formulation for the interior one. \\
 
 In order to introduce the boundary integral operators, which arise from the boundary integral formulation, we consider a simpler problem, which consists of either an interior or an exterior problem only, depending on $K = \{\Omega, \Omega^\mathrm{e}\}$. Let $\bu \in \hcurlcurlK \cap \hdivK$ such that
\begin{subequations}\label{prob:intExtCurlCurl}
\begin{align} 
 \bcurl \bcurl \bu &= \boldsymbol{0} &&\text{in } K, \label{eq:intExtCurlCurlDomain}\\
\Div \bu &= 0 &&\text{in } K, \\
%\dtrace \bu &= \bu_0, &&\text{on } \partial K,\\
 \bu &= \mathcal{O}\left(\abs{\bx^{-1}}\right) && \text{for } \abs{\bx}\rightarrow \infty, \text{ if } K = \Omega^\mathrm{e}.
\end{align}    
 \end{subequations}
Similarly to \cite [Section~3]{hiptmair2005coupled}, we introduce the relevant potential operators that are needed for a Stratton-Chu-like representation of the solution, i.e., linear operators that map boundary data to smooth functions in $\R^3\backslash\Gamma$. We denote by 
\begin{equation*}
	\Psi_{\mathrm{SL}}(\varphi)(\bx) = \int_\Gamma G(\bx,\by) \varphi(\by) \opd \sigma_\by, \quad \bx\notin \Gamma
\end{equation*}
the scalar single-layer potential, by 
\begin{equation*}
\bPsi_{\mathrm{SL}}(\bphi)(\bx) = \int_\Gamma G(\bx,\by)  \bphi(\by) \opd \sigma_\by, \quad \bx\notin \Gamma
\end{equation*}
the vectorial single-layer potential, and by 
\begin{equation*}
\bPsi_{\mathrm{DL}}(\bv)(\bx) = \bcurl_\bx \bPsi_{\mathrm{SL}}\left(\bv \times \bn \right)(\bx) 
\end{equation*}
the Maxwell double-layer potential. Thereby, $G(\bx,\by) = \frac{1}{4\pi}\frac{1}{\abs{\bx-\by}}$ coincides with the fundamental solution of the Laplace equation. With this, 
 the solution $\bu$ of \eqref{prob:intExtCurlCurl} admits the following representation formula:
 \begin{equation}\label{eq:representationFormula}
 	\text{For } \bx \in K \text{ and }\alpha = 0,1, \quad  \bu(\bx) =  (-1)^\alpha\left(\bPsi_\mathrm{SL}(\ntrace \bu)(\bx) + \bPsi_\mathrm{DL}(\dtrace \bu)(\bx) + \nabla\Psi_\mathrm{SL}(\gamma_\bn \bu)(\bx)\right),
 \end{equation} 
% \begin{align}\label{eq:representationFormula}
% \begin{split}
%% \bue(\bx) = -\boldsymbol{\Psi}_\mathrm{SL}(\ntracex \bu)(\bx) - \boldsymbol{\Psi}_\mathrm{DL}(\dtracex \bu)(\bx) - \nabla\Psi_\mathrm{SL}(\gamma_\bn \bu)(\bx),
%\bu(\bx) = (-1)^\alpha(\int_\Gamma G(\bx,\by)  \ntrace \bu(\by) \opd \sigma_\by + \bcurl_\bx \int_\Gamma &G(\bx,\by)\twtrace \bu(\by) \opd \sigma_\by \\&+ \nabla\int_\Gamma G(\bx,\by)  \gamma_\bn \bu(\by) \opd \sigma_\by ), \,\alpha = 0,1,
%\end{split}
% \end{align}
 where $\alpha = 0$ leads to an interior problem, and $\alpha = 1$ to an exterior one. % Note that for vector fields that are in $\bWe$, we see that the contribution of the last term vanishes due to the imposed zero normal trace. \\ 
 For notational consistency, if $K = \Omega^\mathrm{e}$, we replace $\bu$ by $\bue$, $(\dtrace,\ntrace,\gamma_\bn)$ by $(\dtracex,\ntracex,\gamma^\mathrm{e}_\bn)$, and $\boldsymbol{H}$ by $\boldsymbol{H}_\mathrm{loc}$. 
 \begin{remark}
 The condition $\Div \bu = 0$ is necessary for the derivation of the representation formula \eqref{eq:representationFormula}, as mentioned in \cite[Section~5]{hiptmairSymmetric}. However, we aim to relax this condition in the remaining of this work by considering suitable spaces for the solution. 
 \end{remark}
% Transferred to the boundary, the gauging \eqref{eq:gauging} translates to $\Div \ntrace \bu = 0$ and $\gamma_\bn \bu = 0$. This motivates the consideration of the following subspace for the Neumann data
% \begin{equation*}
% \hdivzg := \{ \bv \in \hdivg: \Div_\Gamma \bv = 0 \text{ and } \gamma_\bn \bv = 0\}.
%\end{equation*} 
  Applying the trace operators to \eqref{eq:representationFormula}, and using the jump relations, as given in  \cite [Theorem~3.4]{hiptmair2005coupled}, for instance, yields the following boundary integral equations: For $\alpha = 0$, it holds that
 \begin{subequations}\label{eq:intBIE}
 \begin{align} 
 \dtrace \bu &= \SL \ntrace \bu + (\frac{1}{2} + \DL) \dtrace \bu + \nabla_\Gamma \V_0(\gamma_\bn \bu) \label{eq:intBIE1} ,\\
 \ntrace \bu &= (\frac{1}{2} + \DLA) \ntrace \bu + \HS \dtrace \bu ,\label{eq:intBIE2}
  \end{align} 
  \end{subequations}
  and for $\alpha = 1$
  \begin{subequations}\label{eq:extBIE}
  \begin{align} 
\dtracex \bue &= -\SL \ntracex \bue + (\frac{1}{2} - \DL) \dtracex \bue - \nabla_\Gamma \V_0(\gamma^\mathrm{e}_\bn \bue) \label{eq:extBIE1} ,\\ %&\text{ for } \bx \in \Gamma, \label{eq:bieext}\\
 \ntracex \bue &= (\frac{1}{2} - \DLA) \ntracex \bue - \HS \dtracex \bue , \label{eq:extBIE2}%&\text{ for } \bx \in \Gamma,
  \end{align} 
 \end{subequations}
  where 
  \begin{align*}
  \SL &: \hdivg \rightarrow \hcurlg,\quad \DL: \hcurlg \rightarrow \hcurlg,\\ 
  \DLA &: \hdivg \rightarrow \hdivg, \quad \HS: \hcurlg \rightarrow \hdivg
  \end{align*}
  define continuous and linear boundary integral operators. Note that $\V_0$ is the standard single layer operator from the scalar case, see \cite{sauter}, for instance. Equations \eqref{eq:intBIE} and \eqref{eq:extBIE} can be rewritten in matrix form as 
  \begin{equation*}
  \begin{pmatrix}
  \dtrace \bu \\ \ntrace \bu
  \end{pmatrix} = \left( \frac{1}{2}\widetilde{\Id} + \Cal \right) \begin{pmatrix}
  \dtrace \bu \\ \ntrace \bu \\ \gamma_\bn \bu
  \end{pmatrix} , \quad \begin{pmatrix}
  \dtracex \bue \\ \ntracex \bue
  \end{pmatrix} =\left( \frac{1}{2}\widetilde{\Id} - \Cal \right)  \begin{pmatrix}
  \dtracex \bue \\ \ntracex \bue \\ \gamma_\bn^\mathrm{e} \bue
  \end{pmatrix} ,
  \end{equation*}
  where $\widetilde{\Id} =\left( \Id \,\vert \,\boldsymbol{0}\right) $ denotes the augmented identity operator, and
  \begin{equation*}
  \Cal = \begin{pmatrix}
  \DL & \SL &  \nabla_\Gamma \V_0 \\ \HS &  \DLA & 0
  \end{pmatrix}
  \end{equation*}
a Calder\'on-like operator. In order to extract the Calder\'on projector, we follow \cite[Section~6.1]{KurzDF}, and project \eqref{eq:intBIE1} and \eqref{eq:extBIE1} onto the quotient space $[\hcurlg]= \hcurlg / \nabla_\Gamma \spht$, which consists of the equivalence classes
\begin{equation*}
[\bpsi] := \{  \bpsi + \bxi : \bxi \in \nabla_\Gamma \spht \}, \quad \bpsi \in \hcurlg.
\end{equation*}
We equip it with the norm 
\begin{equation*}
\norm{[\bpsi]}_{[\hcurlg]} := \inf_{\bxi \in \nabla_\Gamma \spht }\{ \norm{\bpsi + \bxi}_\hcurlg \}.
\end{equation*}
In the next Proposition, we give further possible characterizations of the quotient space $[\hcurlg]$. For this, let us first define the following subspace of $\hdivg$
\begin{equation*}
\hdivzg := \{ \bpsi \in \hdivg: \Div_\Gamma \bpsi = 0\}
\end{equation*} 
as the space of solenoidal surface divergence vector fields. Then, we state the following:
\begin{proposition}\label{proposition:representative}
Let $\Omega$ be a bounded Lipschitz domain, which is simply-connected with connected boundary $\Gamma$. Then:
\begin{itemize}
\item The quotient space $[\hcurlg]$ and the orthogonal complement of $\nabla_\Gamma \spht$ in $\hcurlg$ are isometrically isomorphic. We denote $[\hcurlg] \cong (\nabla_\Gamma \spht)^\perp$.
\item The dual space of $[\hcurlg]$ can be identified with the subspace $\hdivzg$.
\end{itemize}
\end{proposition}
\begin{proof}
Let $\mathsf{P}_1: \hcurlg \rightarrow [\hcurlg]$ be a projection, such that for all $\bxi\in \hcurlg$, $\mathsf{P}_1\bxi = [\bxi] := \bxi + \nabla_\Gamma \spht$. Moreover, since $\nabla_\Gamma \spht $ is a closed linear subspace of $\hcurlg$, the projection theorem, see e.g. \cite[Theorem~2.9]{monk2003finite} states that the direct sum decomposition $\hcurlg = (\nabla_\Gamma \spht)^\perp \oplus \nabla_\Gamma \spht$, such that $(\nabla_\Gamma \spht)^\perp \cap \nabla_\Gamma \spht = \{0\}$ holds. Hence, for every $\bxi \in \hcurlg$, there exists unique $\bpsi \in (\nabla_\Gamma \spht)^\perp$ and $\bvphi \in \nabla_\Gamma \spht$, such that $\bxi = \bpsi + \bvphi$. Consequently, there exists a canonical isomorphism $\iota$ of $(\nabla_\Gamma \spht)^\perp$ onto $[\hcurlg]$, which means that the quotient space $[\hcurlg]$ can be identified with the orthogonal complement of $\nabla_\Gamma \spht$ in $\hcurlg$. Moreover, the composition of the projection $\mathsf{P}_1$ with the inverse isomorphism $\iota^{-1}$ corresponds to the orthogonal projection $\mathsf{P}:\hcurlg \rightarrow (\nabla_\Gamma \spht)^\perp$. Since $\ker \mathsf{P} = \nabla_\Gamma \spht $, it follows \[ \norm{[\mathsf{P}\bpsi]}_{[\hcurlg]} = \inf_{\bxi \in \nabla_\Gamma \spht }\{ \norm{\mathsf{P}\bpsi + \mathsf{P}\bxi}_\hcurlg \} =  \norm{\mathsf{P}\bpsi}_\hcurlg .\]
Therefore, $[\hcurlg] \cong (\nabla_\Gamma \spht)^\perp$. \\ 
The result of the second proposition follows from the universal property of the quotient, which gives us a tool to easily construct a natural isomorphism $\hdivzg \rightarrow [\hcurlg]$, since $\hdivg$ and $\hcurlg$ are dual, and $\nabla_\Gamma \spht = \ker \hcurlg$ holds for trivial topologies. The same result follows also directly by using the first proposition, and the first identity in \eqref{eq:diffOperator}. However, the first approach holds even for non-trivial topologies since it only requires $\nabla_\Gamma \spht \subseteq \ker \hcurlg$.
%Then, using the isomorphism $\iota:(\nabla_\Gamma \spht)^\perp \rightarrow [\hcurlg]$, with $\iota(\bpsi) = \bpsi + \nabla_\Gamma \spht$ yields $\iota^{-1} \circ \mathbb{P}: $ a unique identification  we see that there exists an isomorphism of $(\nabla_\Gamma \spht)^\perp$ on to  it follows that we can identify the quotient space $[\hcurlg]$ with the orthogonal complement of $\nabla_\Gamma \spht$ in $\hcurlg$. 
\end{proof}

By taking into account the commutativity of the de~Rham complex, see Figure \ref{fig:deRham}, the identity 
\begin{equation*}
\Div_\Gamma\twtrace \bw = \gamma_\bn \bcurl \bw
\end{equation*}  
holds for all $\bw \in \hcurlK$. In particular, by choosing $\bw = \bcurl \bv \in \hcurlcurlK$ that additionally satisfies \eqref{eq:intExtCurlCurlDomain}, we obtain  
\begin{equation*}
\Div_\Gamma\ntrace \bv = \gamma_\bn \bcurl \bcurl \bv= 0.
\end{equation*} 
%Taking into account that for all $\bv \in \hcurlcurlK$ \revisionme{that satisfy \eqref{eq:intExtCurlCurlDomain},
%\begin{equation*}
%\Div_\Gamma\ntrace \bv = \gamma_\bn \bcurl \bcurl \bv= 0,
%\end{equation*} 
Therefore, \eqref{eq:intBIE2} and \eqref{eq:extBIE2} have to be set in $\hdivzg$, which again can be identified with the dual space of $[\hcurlg]$.
In a weak sense, an equivalent weak Calder\'on projector can be obtained by using the identity 
  \begin{equation*}
  \duality{\bpsi, \nabla_\Gamma \V_0(\gamma_\bn \bu)}_\Gamma = -\duality{\Div_\Gamma \bpsi, \V_0(\gamma_\bn \bu)}_\Gamma = 0,  
  \end{equation*}
which holds whenever $\Div_\Gamma \bpsi = 0$. Hence, by testing \eqref{eq:intBIE1} and \eqref{eq:extBIE1} with $\bpsi \in \hdivzg$, we observe that the contribution of the last term in \eqref{eq:representationFormula} also vanishes. Thus, we obtain for $\alpha = 0$ 
  \begin{equation} \label{eq:bieintvar}
 \duality{\bpsi, \SL \ntrace \bu}_\Gamma - \duality{\bpsi , (\frac{1}{2} - \DL) \dtrace \bu}_\Gamma= 0,  
  \end{equation} 
  and for $\alpha = 1$
  \begin{equation} \label{eq:bieextvar}
 \duality{\bpsi, \SL \ntracex \bue}_\Gamma + \duality{\bpsi , (\frac{1}{2} + \DL) \dtracex \bue}_\Gamma= 0
  \end{equation} 
  for all $\bpsi \in \hdivzg$.
 Moreover, we know from \cite[Theorem~4.4~($\kappa = 0$)]{hiptmair2005coupled} that $\SL$ is symmetric and $\hdivzg$-elliptic, i.e., there exists $C_\SL > 0$ such that 
 \begin{equation}\label{eq:ellipticitySL}
 \duality{\bpsi, \SL \bpsi}_\Gamma \geq C_\SL \norm{\bpsi}_\hdivg^2 \qquad \text{ for all } \bpsi \in \hdivzg.
 \end{equation}
 Furthermore, due to 
  \begin{equation} \label{eq:ellipticityHS}
\duality{\bv, \HS \bv}_\Gamma = \duality{\V_0 \curl_\Gamma \bpsi,\curl_\Gamma \bpsi}_\Gamma \geq C_\HS \norm{\bpsi}^2_\hcurlg \qquad\text{ for all } \bpsi \in [\hcurlg]
 \end{equation}
with $C_\HS>0$, we establish symmetry and ellipticity of $\HS$ in $[\hcurlg]$. We used thereby \cite[Remark~4.26~($\kappa = 0$)]{engleder} and the well-known $\dht$-ellipticity and symmetry of $\V_0$. Note that $\HS$ is only positive semi-definite, if we consider the whole space $\hcurlg$. Moreover, it holds for $\bpsi \in \hdivg$ and $\bxi \in \hcurlg$ that
 \begin{equation}\label{eq:adjointDL}
 \duality{\bpsi, \DL \bxi}_\Gamma = -\duality{ \DLA \bpsi, \bxi}_\Gamma,
 \end{equation} 
cf. \cite[Theorem~3.9~($\kappa = 0$)]{hiptmairSymmetric}. \\

By using the $\hdivzg$-ellipticity of $\SL$, we define in a $\hdivzg$-sense the interior and exterior Steklov-Poincar\'e operators $\Steklov,\,\Steklove: \hcurlg \rightarrow \hdivzg$ starting from \eqref{eq:intBIE1} and \eqref{eq:extBIE1} by
 \begin{subequations}
\begin{align}
 \ntrace \bu & =\SL^{-1}(\frac{1}{2} - \DL) \dtrace \bu := \Steklov \dtrace \bu ,\label{eq:SteklovIntNonSym}\\
  \ntracex \bue &:=  -\SL^{-1}(\frac{1}{2} + \DL) \dtracex \bue := \Steklove \dtracex \bue.\label{eq:SteklovExtNonSym} 
\end{align}
 \end{subequations}
Moreover, inserting the above in \eqref{eq:intBIE2} and \eqref{eq:extBIE2} leads to a symmetric representation of the Steklov-Poincar\'e operators:
 \begin{subequations}
\begin{align}
 \Steklov &= \left( \frac{1}{2} + \DLA \right) \SL^{-1}\left(\frac{1}{2} - \DL\right) + \HS , \label{eq:steklovIntSymm}\\
  \Steklove &= -\left( \frac{1}{2} - \DLA \right) \SL^{-1}\left(\frac{1}{2} + \DL\right) - \HS. 
\end{align}
 \end{subequations}
With this, together with the $\hdivzg$-ellipticity of $\SL$, we see that the Steklov-Poincar\'e operators inherit the properties of the hyper-singular operator $\HS$. In particular, we notice that $\Steklov,\,\Steklove$ are only semi-definite, if we consider the entire space $\hcurlg$, since surface gradient fields lie in the kernel of $\HS$. However, ellipticity is recovered, when restricting the definition space to $[\hcurlg]$, by recovering the ellipticity of $\HS$ in $[\hcurlg]$.  \\
  
As a last preliminary step, we need to find a weak formulation for the interior part, represented by \eqref{eq:curlcurlInt}. This is merely achieved by applying the second Green's identity \eqref{eq:Green2}. However, due to the infinite-dimensional kernel of the $\bcurl$-operator, we observe that $\scalar{\M\bcurl\bu,\bcurl \bv }$ fails to be strongly monotone in the entire $\hcurl$ space. Indeed, a solution of the interior problem can only be unique up to a gradient field. Therefore, we consider a formulation of the interior problem in the quotient space $[\hcurl]:=\hcurl / \nabla\bh$, which is defined accordingly to $[\hcurlg]$. Thus, we equip $[\hcurl]$ with the quotient norm 
\begin{equation*}
\norm{[\bv]}_{[\hcurl]} := \inf_{\bw \in \nabla H^1(\Omega) }\{ \norm{\bv+ \bw}_\hcurl \}.
\end{equation*}
In the following, we give a characterization of the quotient space $[\hcurl]$.
\begin{proposition}\label{prop:characHcurl}
	Let $\Omega$ be a bounded Lipschitz domain, which is simply-connected. Then:
	\begin{itemize}
		\item The quotient space $[\hcurl]$ and the orthogonal complement of $\nabla H^1(\Omega)$ in $\hcurl$ are isometrically isomorphic. We denote $[\hcurl] \cong (\nabla H^1(\Omega))^\perp$.
		\item The space $[\hcurl]$ can be identified with $\hcurl \cap \hdivzh$, where
		\begin{equation*}
		\hdivzh := \{ \bv \in \hdiv: \Div \bv = 0 \text{ and } \gamma_\bn \bv = 0 \}.
		\end{equation*} 
	\end{itemize}
\end{proposition}
\begin{proof}
Using similar arguments as in the first result of Proposition \ref{proposition:representative}, it follows that every element of $[\hcurl]$ can be uniquely represented by a representative $\bv^\perp \in \hcurl$ that lies in the orthogonal complement of $\nabla H^1(\Omega)$ in $\hcurl$. This can be obtained by means of the orthogonal projection $\mathsf{P}: \hcurl \rightarrow (\nabla H^1(\Omega))^\perp$, which can be defined analogously to the one given in the proof of Proposition \ref{proposition:representative}. Hence, $[\hcurl] \cong (\nabla H^1(\Omega))^\perp$ can be established, and $\norm{[\mathsf{P}\bv]}_{[\hcurl]} \cong \norm{\mathsf{P}\bv}_\hcurl$ holds for $\mathsf{P}\bv \in [\hcurl]$.\\
For the second assertion, we first characterize the orthogonal complement of $\nabla H^1(\Omega)$ in $\bl$. For this, we use the Green's identity \cite[Theorem~3.24]{monk2003finite}
\begin{equation*}
	\bw \in \hdiv, \, \varphi \in H^1(\Omega):\quad \scalar{\bw,\nabla \varphi} = \duality{\gamma_\bn \bw, \gamma_0 \varphi}_\Gamma - \scalar{\Div \bw , \varphi}.
\end{equation*}   
Hence, for $\bw$ to satisfy orthogonality in the $\bl$ sense for all $\varphi \in H^1(\Omega)$, we require $\Div \bw = 0$ and $\gamma_\bn \bw = 0$ to hold. Thus, $\bw \in \hdivzh$. Therefore, it follows that $(\nabla H^1(\Omega))^\perp$ in $\hcurl$ can be identified with $\hcurl \cap \hdivzh$, and by the first assertion, the same holds for $[\hcurl]$. 
\end{proof}
%Using similar arguments as in the first result of Proposition \ref{proposition:representative}, it follows that every element of $[\hcurl]$ can be uniquely represented by a representative $\bv^\perp \in \hcurl$ that lies in the orthogonal complement of $\nabla H^1(\Omega)$ in $\hcurl$. This can be obtained by means of the orthogonal projection $\mathsf{P}: \hcurl \rightarrow (\nabla H^1(\Omega))^\perp$, which can be defined analogously to the one given in the proof of Proposition \ref{proposition:representative}. Hence, $[\hcurl] \cong (\nabla H^1(\Omega))^\perp$ can be merely established, and $\norm{[\mathsf{P}\bv]}_{[\hcurl]} \cong \norm{\mathsf{P}\bv}_\hcurl$ holds for $\mathsf{P}\bv \in [\hcurl]$. Moreover, the unique representative $\bv^\perp = \mathsf{P}\bv \in (\nabla H^1(\Omega))^\perp$ satisfies 
%\begin{equation}\label{eq:divfreeCurl}
%\scalar{\mathsf{P} \bv, \nabla \varphi} = 0, \quad \bv \in \hcurl,
%\end{equation}
%for all $\varphi \in \bh$. Analogously to the second proposition in \ref{proposition:representative}, the dual space of $[\hcurl]$, which is denoted by $[\hcurl]^\prime$ can be identified with the subspace 
%$\hdivzh$ given by 
%\begin{equation*}
%\hdivzh := \{ \bv \in \hdiv: \Div \bv = 0 \text{ and } \gamma_\bn \bv = 0 \}.
%\end{equation*} 
\begin{remark}\label{remark:notation}
For the sake of notational simplicity, we will be considering representatives of the equivalence classes as discussed above but we will keep the notation of elements of the quotient spaces. Moreover, every $\bv \in [\hcurl]$ and $\bpsi \in [\hcurlg]$ will be implicitly identified with a $\mathsf{P\bw} \in (\nabla H^1(\Omega))^\perp,\,\bw \in \hcurl$, and $\mathsf{P\bxi} \in (\nabla_\Gamma \spht)^\perp,\,\bxi \in \hcurlg$, respectively.
\end{remark}
Note that due to the properties of the trace operators, the definitions given in this section can be adapted accordingly to the new setting. For example, it holds that
 \begin{equation*}
 \dtrace: [\hcurl] \rightarrow [\hcurlg]
 \end{equation*}
is linear, continuous, and surjective. Moreover, due to the Green's identity \eqref{eq:Green2}, it holds that $\duality{\ntrace \bu, \dtrace \bv}_\Gamma$ defines a non-degenerate duality pairing on $\hdivzg \times [\hcurlg]$. Hence, the skew-adjoint relation between $\DL$ and $\DLA$ in \eqref{eq:adjointDL} is inherited in $\hdivzg \times [\hcurlg]$.
\section{non-symmetric FEM-BEM coupling} \label{sec:fem-bem}
In this next section, we give a non-symmetric FEM-BEM coupling, and show well-posedness of the arising problem, along with a stability result for a special type of practice-oriented non-linearity. \\
Let us consider Problem \ref{prob:strongForm}. Testing \eqref{eq:curlcurlInt} with $\bv \in [\hcurl]$, applying the Green's identity \eqref{eq:Green2}, and inserting \eqref{eq:NeumannJump} leads to
\begin{equation}\label{eq:FEMpart}
 \scalar{\M\bcurl\bu,\bcurl \bv }- \duality{\ntracex \bue, \dtrace \bv}_\Gamma = \scalar{\boldsymbol{f} ,\bv} + \duality{\boldsymbol{\phi}_0, \dtrace \bv}_\Gamma. 
\end{equation}
Inserting \eqref{eq:DirichletJump} in \eqref{eq:bieextvar} yields together with \eqref{eq:FEMpart} to the following non-symmetric coupling:
\begin{problem}\label{problem:varProb}
 Find $\uvec: = (\bu,\bphi:= \ntracex \bue) \in \bH := [\hcurl] \times \hdivzg$ such that 
\begin{align*}
 \scalar{\M\bcurl\bu,\bcurl \bv }- \duality{\bphi, \dtrace \bv}_\Gamma &= \scalar{\boldsymbol{f} ,\bv} + \duality{\boldsymbol{\phi}_0, \dtrace \bv}_\Gamma, \\ 
 \duality{\bpsi, \SL \bphi}_\Gamma + \duality{\bpsi , (\frac{1}{2} + \DL) \dtrace \bu}_\Gamma &= \duality{\bpsi , (\frac{1}{2} + \DL) \boldsymbol{u}_0}_\Gamma
\end{align*}  
for all $\vvec:= (\bv,\bpsi) \in \bH$, and $\left( \boldsymbol{f},\bphi_0, \bu_0  \right) \in \left([\hcurl]^\prime  \times\hdivzg \times [\hcurlg] \right)$.
\end{problem}
\noindent Thereby, $[\hcurl]^\prime$ denotes the dual space of $[\hcurl]$. A concrete characterization can be reached by means of the inner product of $[\hcurl]$. Thus, $[\hcurl]^\prime \cong [\hcurl]$, i.e., we identify the dual of $[\hcurl]$ with the space itself. This leads by Proposition \ref{prop:characHcurl} to the consistency condition
	\begin{equation*}
		\Div \boldsymbol{f} = 0,\quad \gamma_\bn \boldsymbol{f} = 0.
	\end{equation*}
Moreover, recall that $\hdivzg$ and $[\hcurlg]$ define a non-degenerate duality pairing. For notational purposes, we refer again to Remark \eqref{remark:notation}. \\
For convenience, we equip the product space $\bH:= [\hcurl] \times \hdivzg $ with the following norm
\begin{equation}\label{eq:normH}
\norm{\vvec}_\bH := \left( \norm{\bv}_{\hcurl}^2 + \norm{\bpsi}^2_\hdivg \right)^{\frac{1}{2}} \qquad \text{ for } \vvec := (\bv,\bpsi) \in  \bH,
\end{equation} 
and write the variational formulation derived above in a compact form: 
\begin{problem}\label{problem:magnetostaticVF} Find $\uvec: = (\bu,\bphi) \in \bH$ such that 
$\ba (\uvec , \vvec) = \bell (\vvec)$ for all $\vvec: = (\bv,\bpsi) \in \bH$. Thereby, 
\begin{equation*}
\ba (\uvec , \vvec) := \scalar{\M\bcurl\bu,\bcurl \bv }- \duality{\bphi, \dtrace \bv}_\Gamma + \duality{\bpsi, \SL \bphi}_\Gamma + \duality{\bpsi , (\frac{1}{2} + \DL) \dtrace \bu}_\Gamma
\end{equation*}
is a linear form (linear in the second argument), and 
\begin{equation*}
\bell(\vvec) := \scalar{\boldsymbol{f} ,\bv} + \duality{\boldsymbol{\phi}_0, \dtrace \bv}_\Gamma + \duality{\bpsi , (\frac{1}{2} + \DL) \boldsymbol{u}_0}_\Gamma
\end{equation*}
is a linear functional.
\end{problem} 
In the subsequent analysis, we use standard results for strongly monotone operators \cite{zeidler}. For this, let $\Astab: \bH \rightarrow \bH^\prime$ be the induced non-linear operator, with $\bH^\prime$ denoting the dual space of $\bH$, which arises from
\begin{equation}\label{eq:nonLinearOp}
\duality{\Astab (\uvec), \vvec} = \ba(\uvec,\vvec) \qquad \forall \uvec,\vvec \in \bH.
\end{equation}
% Starting from the boundary integral equation \eqref{eq:bieint}, we define an analog to the interior Steklov-Poincar\'e operator, which maps Dirichlet data to Neumann data
%  \begin{equation} \label{eq:steklov}
% \ntrace \bv = \Steklov \dtrace \bv , \quad \Steklov := \SL^{-1}(\frac{1}{2} - \DL).
%  \end{equation} 
%  By using the second BIE in addition, we find a symmetric representation for the Steklov-Poincar\'e operator 
%  \begin{equation}\label{eq:SteklovSymmetric}
%\Steklov := \HS + \left( \frac{1}{2} + \DLA \right) \SL^{-1} \left(\frac{1}{2} - \DL \right). 
%\end{equation}
First, we introduce the following contraction based result, which will be used in the monotonicity estimate.
\begin{lemma}\label{theorem:contraction}
Let $\Steklov:[\hcurlg] \rightarrow \hdivzg$ be an interior Steklov-Poincar\'e operator. There exists a contraction constant $C_\DL := \frac{1}{2}+\sqrt{\frac{1}{4}-C_\SL C_\HS} \in [\frac{1}{2},1)$ such that 
\begin{equation*}
(1-C_\DL) \norm{\bpsi }_{\SL^{-1}} \leq \norm{\left(\frac{1}{2} - \DL \right) \bpsi }_{\SL^{-1}} \leq C_\DL \norm{\bpsi }_{\SL^{-1}},  \quad \bpsi \in [\hcurlg],
\end{equation*}
and
\begin{equation*}
\norm{\left(\frac{1}{2} - \DL \right) \bpsi }_{\SL^{-1}}^2 \leq C_\DL \duality{\Steklov \bpsi,\bpsi }_\Gamma, \quad \bpsi \in [\hcurlg].
\end{equation*}
Thereby, $C_\SL$ and $C_\HS$ are the ellipticity constants given in \eqref{eq:ellipticitySL} and \eqref{eq:ellipticityHS}, respectively, and $\norm{\cdot}_{\SL^{-1}} = \duality{\SL^{-1} \cdot,\cdot}_\Gamma^\frac{1}{2}$.  
\end{lemma}
\begin{proof}
The first assertion follows analogously to the scalar case in \cite{SteinbachContraction}. For convenience, we sketch the main steps of the proof. Using the symmetric representation of the interior Steklov-Poincar\'e operator, given in \eqref{eq:steklovIntSymm}, together with the $\hdivzg$-ellipticity of $\SL$, it holds 
\begin{equation}\label{eq:term}
\begin{split}
\norm{\left(\frac{1}{2} - \DL \right) \bpsi }_{\SL^{-1}}^2 &= \duality{\SL^{-1}\left(\frac{1}{2} - \DL \right) \bpsi ,\left(\frac{1}{2} - \DL \right) \bpsi}_\Gamma \\
&= \duality{\Steklov \bpsi,\bpsi}_\Gamma - \duality{\HS \bpsi,\bpsi}_\Gamma. 
\end{split}
\end{equation}
Using the same argumentation as in \cite[Proposition~5.4]{SteinbachContraction}, we rely on \cite[Theorem~12.33]{rudin}, which states that there exists a self-adjoint square root $\SL^{\frac{1}{2}}$ of $\SL$, which is also invertible, since $\SL$ is a self-adjoint and invertible operator in $\hdivzg$. We denote the inverse of the square root operator by $\SL^{-\frac{1}{2}} $. As a consequence, it follows that $\SL^{-\frac{1}{2}} = \SL^{\frac{1}{2}} \SL^{-1}$ and $\norm{\SL^{-\frac{1}{2}} \bv}_{\boldsymbol{L}^2(\Gamma)} = \norm{\bv}_{\SL^{-1}}$. With this, and following the steps of \cite[Theorem~5.1]{SteinbachContraction}, we get for the first term of the right-hand side of \eqref{eq:term}
\begin{equation}\label{eq:firstTerm}
\begin{split}
\duality{\Steklov \bpsi,\bpsi}_\Gamma &= \duality{\SL^{-\frac{1}{2}} \SL \Steklov \bpsi, \SL^{-\frac{1}{2}} \bpsi}_\Gamma \\
 &\leq \norm{\left(\frac{1}{2} - \DL \right) \bpsi }_{\SL^{-1}} \norm{\bpsi}_{\SL^{-1}},
\end{split}
\end{equation}
where we used the first identity of $\Steklov$ given in \eqref{eq:SteklovIntNonSym}.
For the second term, we need the following result 
\begin{equation*}
C_\SL \duality{\SL^{-1} \bpsi,\bpsi}_\Gamma \leq \norm{\bpsi}_{\hcurlg}^2,
\end{equation*}
which can be derived as in \cite[Proposition~5.2]{SteinbachContraction} by a duality argument, and by making use of the bijectivity of $\SL$ in $\hdivzg$. Then, it follows from the $[\hcurlg]$-ellipticity of $\HS$ \eqref{eq:ellipticityHS} that
\begin{equation}\label{eq:secondTerm}
\begin{split}
\duality{\HS \bpsi,\bpsi}_\Gamma &\geq C_\HS C_\SL \duality{\SL^{-1} \bpsi,\bpsi}_\Gamma \\ &=  C_\HS C_\SL \norm{\bpsi}_{\SL^{-1}}^2.
\end{split}
\end{equation}
Inserting \eqref{eq:firstTerm} and \eqref{eq:secondTerm} in \eqref{eq:term} yields 
\begin{equation*}
\norm{\left(\frac{1}{2} - \DL \right) \bpsi }_{\SL^{-1}}^2  \leq\norm{\left(\frac{1}{2} - \DL \right) \bpsi }_{\SL^{-1}} \norm{\bpsi}_{\SL^{-1}} -C_\HS C_\SL \norm{\bpsi}_{\SL^{-1}}^2.
\end{equation*}
Therefore, solving the second order inequality leads to the assertion
\begin{equation*}
(1-C_\DL) \norm{\bpsi}_{\SL^{-1}} \leq \norm{\left(\frac{1}{2} - \DL \right) \bpsi }_{\SL^{-1}}  \leq C_\DL \norm{\bpsi}_{\SL^{-1}}
\end{equation*}
with $C_\DL := \frac{1}{2}+\sqrt{\frac{1}{4}-C_\SL C_\HS}$. Moreover, it  follows clearly that $C_\DL$ lies in $[\frac{1}{2},1)$. \\
The second assertion follows similarly to \cite[Lemma~2.1]{of2014ellipticity}. Equation \eqref{eq:term} reads 
\begin{equation}
 \duality{\Steklov \bpsi,\bpsi}_\Gamma  = \norm{\left(\frac{1}{2} - \DL \right) \bpsi }_{\SL^{-1}}^2 + \duality{\HS \bpsi,\bpsi}_\Gamma . 
\end{equation}
The second term of the right-hand side can be estimated by \eqref{eq:secondTerm} and the first assertion. We arrive at 
\begin{equation*}
\duality{\HS \bpsi,\bpsi}_\Gamma \geq \frac{C_\SL C_\HS}{C_\DL^2}\norm{\left(\frac{1}{2} - \DL \right) \bpsi }_{\SL^{-1}}^2.
\end{equation*}
By observing that $ C_\SL C_\HS = C_\DL (1-C_\DL)$, the assertion follows merely.
\end{proof}%
\noindent Then, we establish Lipschitz continuity and strong monotonicity of the non-linear operator $\Astab$. 
\begin{theorem}[Strong monotonicity]\label{theorem:monotonicity}
Let $\Astab: \bH \rightarrow \bH^\prime$ be defined as in \eqref{eq:nonLinearOp} with $\bH:= [\hcurl] \times \hdivzg $. Then,
\begin{itemize}
\item $\Astab$ is Lipschitz continuous, i.e., there exists a constant $C_\mathrm{L}>0$ such that 
\begin{equation*}
\norm{\Astab(\uvec) - \Astab(\vvec)}_{\bH^\prime} \leq C_\mathrm{L} \norm{\uvec - \vvec}_\bH
\end{equation*}
\item if $C_\mathrm{M}^\M > \frac{1}{4} C_\DL$, then $\Astab$ is strongly monotone, i.e., there exists a constant $C_\mathrm{E}>0$ such that 
\begin{equation*}
\duality{\Astab(\uvec) - \Astab(\vvec),\uvec - \vvec} \geq C_\mathrm{E} \norm{\uvec - \vvec}_\bH^2
\end{equation*}
\end{itemize}
for all $\uvec,\vvec \in \bH$.
\end{theorem}
\begin{proof}
The Lipschitz continuity of $\Astab$ follows from the Lipschitz continuity of the operator $\M$, and the continuity of the boundary integral operators. \\
The second assertion follows the same appraoch as in \cite[Lemma~2.5]{EEK20}. Let $\wvec=(\bw,\bxi) :=\uvec - \vvec \in \bH$. Then, 
\begin{equation*}
\duality{\Astab(\uvec) - \Astab(\vvec),\wvec} =  \underbrace{\scalar{\M\bcurl\bu - \M\bcurl \bv, \bw }}_{(I)} - \underbrace{\duality{\bxi , (\frac{1}{2} - \DL) \dtrace \bw}_\Gamma}_{(II)} + \duality{\bxi, \SL \bxi}_\Gamma.
\end{equation*}
First, by assuming the strong monotonicity of $\M$, Assumption \ref{A2} reads 
\begin{equation*}
\scalar{\M\bcurl\bu - \M\bcurl \bv, \bw } \geq C_M^\M \norm{\bcurl \bw}_\bl^2,
\end{equation*}
and holds for all $\bu,\bv \in \hcurl$. Hence, it also remains true for the representatives in $[\hcurl]$.
Next, we consider for a representative $\bw \in [\hcurl]$ the splitting $\bw = \bw_\Gamma + \widetilde{\bw}$, where $\widetilde{\bw} \in [\hzgcurl]:= \{ \bv \in [\hcurl]: \dtrace \bv = \boldsymbol{0} \text{ on } \Gamma \}$, and $\bw_\Gamma$ is the solution of the following boundary value problem (BVP) 
\begin{align*}
\bcurl \bcurl \bw_\Gamma &= \boldsymbol{0} &\text{in } \Omega,\\
\dtrace \bw_\Gamma &= \dtrace \bw &\text{on } \Gamma.  
\end{align*} 
Applying \eqref{eq:Green2} (with $\M = \Id$) to the BVP leads to  
\begin{equation*}
 \scalar{\bcurl\bw_\Gamma,\bcurl \bw_\Gamma } = \duality{\ntrace \bw_\Gamma, \dtrace \bw_\Gamma}_\Gamma = \duality{\Steklov \dtrace \bw_\Gamma, \dtrace \bw_\Gamma}_\Gamma,
\end{equation*}
where $\Steklov$ is the interior Steklov-Poincar\'e operator. Since by construction $\scalar{\bcurl\bw_\Gamma,\bcurl\widetilde{\bw}} = 0$ for all $\widetilde{\bw} \in [\hzgcurl]$, we have
\begin{equation} \label{eq:splitting}
\begin{split}
\norm{\bcurl\bw}_\bl^2 &= \norm{\bcurl\widetilde{ \bw}}_\bl^2 + \norm{\bcurl \bw_\Gamma}_\bl^2 \\
& = \norm{\bcurl\widetilde{ \bw}}_\bl^2 + \duality{\Steklov \dtrace \bw_\Gamma, \dtrace \bw_\Gamma}_\Gamma.
\end{split}
\end{equation} 
Hence, $(I)$ can be estimated by
\begin{equation*}
\scalar{\M\bcurl\bu - \M\bcurl \bv, \bw } \geq C_M^\M \norm{\bcurl\widetilde{ \bw}}_\bl^2 + C_M^\M \duality{\Steklov \dtrace \bw_\Gamma, \dtrace \bw_\Gamma}_\Gamma.
\end{equation*}
To estimate $(II)$, we have 
\begin{equation*}
\begin{split}
\duality{\bxi , (\frac{1}{2} - \DL) \dtrace \bw}_\Gamma &= \duality{\SL\bxi , \SL^{-1}(\frac{1}{2} - \DL) \dtrace \bw}_\Gamma\\
& \leq \norm{\bxi}_\SL \norm{(\frac{1}{2} - \DL) \dtrace \bw}_{\SL^{-1}} \qquad \text{ with } \norm{\bxi}_\SL=<\bxi,\SL\bxi> \\
& \leq \norm{\bxi}_\SL \sqrt{C_\DL \duality{\Steklov \dtrace \bw_\Gamma, \dtrace \bw_\Gamma}_\Gamma},
\end{split}
\end{equation*}
where we used Lemma \ref{theorem:contraction} in the last step. Altogether,
\begin{multline*}
\duality{\Astab(\uvec) - \Astab(\vvec),\bw} \geq C_M^\M \norm{\bcurl\widetilde{ \bw}}_\bl^2 + C_M^\M \duality{\Steklov \dtrace \bw_\Gamma, \dtrace \bw_\Gamma}_\Gamma\\ - \norm{\bxi}_\SL \sqrt{C_\DL \duality{\Steklov \dtrace \bw_\Gamma, \dtrace \bw_\Gamma}_\Gamma} + \norm{\bxi}_\SL^2 ,
\end{multline*}
which can be written in quadratic form as 
\begin{equation*}
\duality{\Astab(\uvec) - \Astab(\vvec),\bw} \geq C_\mathrm{M}^\M \norm{\bcurl\widetilde{\bw}}^2_\bl+ \duality{ \begin{pmatrix}
C_\mathrm{M}^\M & -\frac{1}{2}\sqrt{C_\DL}  \\ -\frac{1}{2}\sqrt{C_\DL} &1   
\end{pmatrix} \textbf{x}, \textbf{x}}
\end{equation*}%
with $\textbf{x} = \begin{pmatrix}
\sqrt{\duality{\Steklov \dtrace \bw_\Gamma, \dtrace \bw_\Gamma}_\Gamma} , \norm{\bxi}_\SL
\end{pmatrix}^\top$. We see that positive definiteness of the quadratic form is guaranteed, if $C_\mathrm{M}^\M > \frac{1}{4} C_\DL$ holds. By using the smallest eigenvalue of the matrix, we reach with $C_\mathrm{stab} = \min\Big\{C_\mathrm{M}^\M,\frac{1}{2}\left(C_\mathrm{M}^\M+1-\sqrt{\left(C_\mathrm{M}^\M-1\right)^2+  C_\DL }\right)\Big\}$ 
\begin{equation}\label{eq:stability}
\begin{split}
\duality{\Astab(\uvec) - \Astab(\vvec),\bw} & \geq C_\mathrm{stab} \left(\norm{\bcurl\widetilde{ \bw}}_\bl^2+\duality{\Steklov \dtrace \bw_\Gamma, \dtrace \bw_\Gamma}_\Gamma + \norm{\bxi}^2_\SL \right) \\
&\overset{\eqref{eq:splitting}}{=} C_\mathrm{stab}\left(\norm{\bcurl\bw}_\bl^2 + \norm{\bxi}^2_\SL \right).
\end{split}
\end{equation}
The remaining step is to show that $\sqrt{\norm{\bcurl \bw}^2_\bl + \norm{\bxi}^2_\SL}$ defines an equivalent norm to $\norm{\wvec}_\bH $. From the properties of $\SL$, we see that $\norm{\bxi}_\SL$ defines an equivalent norm of $\norm{\bxi}_\hdivg$ for $\bxi \in \hdivzg$. This reduces the problem to showing that $\norm{\bcurl \bw}_\bl$ is an equivalent norm to $\norm{\bw}_\hcurl$ for the representatives in $[\hcurl]$. From Proposition \ref{prop:characHcurl}, we know that $[\hcurl]$ can be identified with $\hcurl \cap \hdivzh$. Then, by using the Friedrichs' inequality \cite[Corollary~3.51]{monk2003finite}, there exists $C_\mathrm{F}>0$ such that 
\begin{equation*}
\norm{\bv}_\bl \leq C_\mathrm{F} \norm{\bcurl \bv}_\bl.
\end{equation*}
Therefore, inserting  
\begin{equation*}
\norm{\bw}_\hcurl \leq (1+C_\mathrm{F}) \norm{\bcurl \bw}_\bl, \quad \bw \in [\hcurl]
\end{equation*}
in \eqref{eq:stability} establishes togtether with the norm equivalence of $\norm{\bxi}_\SL$ and $\norm{\bxi}_\hdivg$ for $\bxi \in \hdivzg$ the strong monotonicity of the non-linear operator $\Astab$.
\end{proof}%
%\begin{lemma}[Lipschitz continuity]\label{lemma:lipschitz}
%Let $\Astab: \bH \rightarrow \bH^\prime$ be as defined in \eqref{eq:nonLinearOp}. Then, 
%$\Astab$ is Lipschitz continuous, i.e., there exists a constant $C_\mathrm{L}>0$ such that 
%\begin{equation*}
%\norm{\Astab(\uvec) - \Astab(\vvec)}_{\bH^\prime} \leq C_\mathrm{L} \norm{\uvec - \vvec}_\bH,
%\end{equation*}
%\end{lemma}
%\begin{proof}
%The Lipschitz continuity of $\Astab$ follows from the Lipschitz continuity of the operator $\M$, and the continuity of the boundary integral operators.
%\end{proof}
\noindent Finally, we state the main result of this section.
\begin{theorem} [Well-posedness, \cite{zeidler}]\label{theorem:wellposedness}
Let $\Astab: \bH \rightarrow \bH^\prime$ be as defined in \eqref{eq:nonLinearOp}, and let $C_\mathrm{M}^\M > \frac{1}{4} C_\DL$. Then, for $\Astab$ Lipschitz continuous and strongly monotone, Problem \ref{problem:magnetostaticVF} admits a unique solution, for any $(\boldsymbol{f},\boldsymbol{u}_0,\boldsymbol{\phi}_0) \in [\hcurl]^\prime \times [\hcurlg] \times \hdivzg$. 
\end{theorem}
\begin{proof}
The result is a consequence of the main theorem on strongly monotone operator equations, see \cite[Theorem~25B,~Corollary~25.7]{zeidler}, together with the assertions of Theorem \ref{theorem:monotonicity}.
\end{proof}%
In practical electromagnetic applications, a typical reluctivity tensor, which is modeled here by the non-linear operator $\M$, takes the following form: $\M \bw := g(\abs{\bw }) \bw $ with $g: \R^3 \rightarrow \R$ a non-linear function. In this special case, a stability result can further be obtained. 
\begin{lemma}
Let the non-linear operator $\M $ be of the form $\M \bw := g(\abs{\bw }) \bw $ with $g: \R^3 \rightarrow \R$ a non-linear function. Moreover, we assume that $C_\mathrm{M}^\M > \frac{1}{4} C_\DL$. Then, for the solution $\uvec \in \bH$ of Problem \ref{problem:magnetostaticVF} with the right-hand sides $(\boldsymbol{f},\boldsymbol{u}_0,\boldsymbol{\phi}_0) \in [\hcurl]^\prime \times [\hcurlg] \times \hdivzg$, there exists $C > 0$ such that 
\begin{equation*}
\norm{\uvec}_\bH \leq C \left( \norm{\boldsymbol{f}}_{[\hcurl]^\prime} + \norm{\bu_0}_{\hcurlg} + \norm{\bphi_0}_\hdivg \right). 
\end{equation*}
\end{lemma}
\begin{proof}
Theorem \ref{theorem:monotonicity} states that $\Astab$ is strongly monotone for $C_\mathrm{M}^\M > \frac{1}{4} C_\DL$. Moreover, the boundary integral operator $\DL:[\hcurlg]\rightarrow [\hcurlg]$ is a continuous operator. Hence, there exists $C^\DL>0$, such that $\norm{\DL\bpsi}_{\hcurlg} \leq C^\DL \norm{\bpsi}_{\hcurlg}$, $\bpsi \in [\hcurlg]$. The continuity of the trace operator $\dtrace$ leads also to the existence of $C_\mathrm{tr}>0$, such that $\norm{\dtrace\bv}_{\hcurlg} \leq C_\mathrm{tr} \norm{\bv}_{\hcurl}$, for $\bv \in [\hcurl]$. With this, the proof can be done analogously to \cite[Lemma~2.7]{EEK20}, when removing the stabilization term.
\end{proof}%
The next section is devoted to the discretization framework intended to the approximation of the variational formulation given in Problem \ref{problem:varProb}.

\section{Isogeometric Analysis}\label{sec:iga}
Geometry design in modern CAD softwares rely mostly on B-Splines and their extensions such as NURBS, T-Splines etc... We refer to \cite{cottrell2009isogeometric} for a more detailed introduction to the subject.
The main idea of Isogeometric Analysis is to use this same type of basis functions to construct suitable solution spaces for a given problem. In our case, our continuous setting rely solely on Sobolev spaces that are involved in the de~Rham complex, see Figure \ref{fig:deRham}. In \cite{da2014mathematical} and \cite{buffa2020multipatch}, we find the theory that is necessary for our purpose. Hence, we just give here a brief introduction to B-Splines followed by a full discretization of the de~Rham complex, for convenience. \\   

B-Splines of arbitrary dimensions can be easily obtained starting from the one-dimensional case by means of tensor product relationship. In each dimension, a knot vector $\Xi := \{ \xi_0, \dots, \xi_{m} \},$ with  $m \in \mathbb{N}$ needs to be defined. The vector $\Xi$ determines the dimension, the regularity, and the degrees of the arising B-Spline basis functions. In this work, we only consider $p$-open knot vectors, where $p < m$ denotes the degree of the obtained functions. For a $p$-open knot vector $\Xi$, there is 
\begin{equation*}
0 = \xi_0 = \dots = \xi_{p} < \xi_{p+1} \leq \dots \leq \xi_{m-p-1} < \xi_{m-p} = \dots = \xi_{m} = 1.
\end{equation*}   
We denote by $k := m-p $ the number of B-Spline basis functions associated to $\Xi$, which can now be computed recursively for $p\geq 1$ by 
\begin{equation*}
b_{i}^{p} (x) = \frac{x - \xi_i}{\xi_{i+p} - \xi_i} b_{i}^{p-1}(x) + \frac{\xi_{i+p+1}-x}{\xi_{i+p+1}-\xi_{i+1}} b_{i+1}^{p-1}(x),
\end{equation*}
for all $i=0 \dots k-1$, starting the recurrence from a piecewise constant function, 
 \begin{equation*}
b_{i}^{0}(x) = \begin{cases} 1 & \text{if }  \xi_i \leq x \leq \xi_{i+1} \\ 0 & \text{otherwise}  \end{cases}.
\end{equation*}
For some $d$-dimensional domain, we introduce  $\bXi = \Xi_1 \times \cdots \times \Xi_d$ as a multi-dimensional knot vector. It is given by a Cartesian product of $p$-open knot vectors. Equivalently to the one-dimensional case, $\bXi$ induces a vector $\bp \in \mathbb{N}^d$ containing the degrees $\{p_i\}_{i=1\dots d}$ in every parametric direction, and a vector $\bk\in \mathbb{N}^d$ containing the corresponding dimensions $\{k_i\}_{i=1\dots d}$ of the space. Hence, a B-Spline space on a $d$-dimensional domain can be constructed with
\begin{equation*}
\mathbb{S}_{\bp}(\bXi) = \Span{\{b_{\bi}^{\bp} \}},
\end{equation*} 
where $b_{\bi}^{\bp} := \prod_{j=1}^d \{b_{i_j}^{p_j}\}_{i_j=0 \dots k_j-1}$. Starting from the space $\mathbb{S}_{\bp}(\bXi)$, an exact sequence of B-Spline spaces can be defined. Let $\bx=\{x_i\}_{i = 1,\dots,d} \in \R^d$. For the derivative of a B-Spline space with respect to $x_i$, which also consists of B-Spline functions with a reduced knot vector in the $i$-th component, cf. \cite[Equation~2.7]{da2014mathematical}, we adopt the following notation
%\begin{equation*}
%\frac{\opd }{\opd \bx_\balpha}\mathbb{S}_{\bp}(\bXi) = \mathbb{S}_{\bp}^{\balpha}(\bXi_\balpha) = \mathbb{S}_{\bp-\balpha}(\bXi_\alpha),
%\end{equation*}
%where $(\bXi_\balpha)_i = \Xi_i$ if $\alpha_i = 0$, and $(\bXi_\balpha)_i = \Xi^\prime_i$ if $\alpha_i = 1$
\begin{equation*}
\frac{\partial }{\partial x_i}\mathbb{S}_{\bp}(\bXi) = \frac{\partial }{\partial x_i}\mathbb{S}_{(p_1,\dots,p_i,\dots,p_d)}(\Xi_1,\dots,\Xi_i, \dots,\Xi_d) = \mathbb{S}_{(p_1,\dots,p_i-1,\dots,p_d)}(\Xi_1,\dots,\Xi_i^\prime, \dots,\Xi_d),
\end{equation*}
where $\Xi_i^\prime = \{\xi_1,\dots,\xi_{m-1} \}$, see \cite[Section~5.2]{da2014mathematical}. In a three-dimensional parametric space, which is represented by $\bXi$, we introduce the following B-Spline spaces
\begin{equation}\label{eq:BsplineParam}
\begin{aligned}
\mathbb{S}^0_{\bp}(\bXi) &= \mathbb{S}_{(p_1,p_2,p_3)}(\Xi_1,\Xi_2, \Xi_3),\\ 
\mathbb{S}^1_{\bp}(\bXi) &=  \mathbb{S}_{(p_1-1,p_2,p_3)}(\Xi_1^\prime,\Xi_2, \Xi_3) \times \mathbb{S}_{(p_1,p_2-1,p_3)}(\Xi_1,\Xi_2^\prime, \Xi_3) \times \mathbb{S}_{(p_1,p_2,p_3-1)}(\Xi_1,\Xi_2, \Xi_3^\prime) ,\\ 
\mathbb{S}^2_{\bp}(\bXi) &=  \mathbb{S}_{(p_1,p_2-1,p_3-1)}(\Xi_1,\Xi_2^\prime, \Xi_3^\prime) \times \mathbb{S}_{(p_1-1,p_2,p_3-1)}(\Xi_1^\prime,\Xi_2, \Xi_3^\prime) \times \mathbb{S}_{(p_1-1,p_2-1,p_3)}(\Xi_1^\prime,\Xi_2^\prime, \Xi_3),\\ 
\mathbb{S}^3_{\bp}(\bXi) &=  \mathbb{S}_{(p_1-1,p_2-1,p_3-1)}(\Xi_1^\prime,\Xi_2^\prime, \Xi_3^\prime),
\end{aligned} 
\end{equation}
to form an exact sequence \cite[Section~5.2]{da2014mathematical}: 
\begin{equation*}
\mathbb{S}^0_{\bp}(\bXi) \xrightarrow{\nabla} \mathbb{S}^1_{\bp}(\bXi)\xrightarrow{\bcurl} \mathbb{S}^2_{\bp}(\bXi)\xrightarrow{\Div} \mathbb{S}^3_{\bp}(\bXi).
\end{equation*}

As mentioned previously, the same type of basis functions are used for discretization purposes as well as for the representation of the geometry, which we assume here to be a compact orientable manifold $K$ of dimension $d \leq 3$ embedded in a three-dimensional Euclidean space. Moreover, we assume that there exists a regular tessellation $K = \bigcup_{l=0}^{N_\kappa-1} \kappa_l$, such that for all $l_1,l_2 = 0,\dots,N_\kappa-1, \, l_1\neq l_2,$ $\kappa_{l_1} \cap \kappa_{l_2}$ is empty. Thereby, $\{\kappa_l\}_{l=0,\dots,N_\kappa-1}$ is a set of $N_\kappa$ patches, which is given by a family of diffeomorphisms  
$\{\boldsymbol{f}_l : [0,1]^d \rightarrow \kappa_l\}_{l = 0,\dots,N_\kappa-1}$. In the context of Isogeometric Analysis, we define the parametrizations $\{\boldsymbol{f}_l\}_{l = 0,\dots,N_\kappa-1}$ in terms of B-Spline basis functions, namely, 
\begin{equation}\label{eq:parametrization}
\boldsymbol{f}_l(\bx) = \sum_{\bi\in \I_l} \bc_\bi b_{\bi}^{\bp} (\bx),
\end{equation} 
where $\I_l= \{ \bi = (i_1,\dots,i_d): i_j = 0,\dots,k_j - 1 \text{ with } j=1,\dots, d\}$ is a the set of multi-indices corresponding to patch $l$, and $\bc_\bi \in \R^3$ are elements of a set of control points, see \cite{piegl2012nurbs}, for more details about the definition of B-Spline functions and curves. 
Note that the parametrizations $\{\boldsymbol{f}_l\}_{l = 0,\dots,N_\kappa-1}$ are consequently disjoint. With this at hand, we can now define B-Spline spaces in the physical domains by mapping the locally defined B-Splines in the parameter domain to the patches $\kappa_l$ of the considered domain. Depending on the aimed regularity of the spaces, different push-forwards may be applied. In this work, we resort to the application of the following push-forwards
\begin{align*}
\iota_0^{-1}(\boldsymbol{f}_l)(v) &:= v \circ \boldsymbol{f}_l^{-1}, \\
\iota_1^{-1}(\boldsymbol{f}_l)(\bv) &:=(\opd \boldsymbol{f}_l^{\top})^{-1} (\bv \circ \boldsymbol{f}_l^{-1}), \\
\iota_2^{-1}(\boldsymbol{f}_l)(\bv) &:= \frac{(\opd \boldsymbol{f}_l) (\bv \circ \boldsymbol{f}_l^{-1})}{\det(\opd \boldsymbol{f}_l)}, \\
\iota_3^{-1}(\boldsymbol{f}_l)(v) &:= \frac{v \circ \boldsymbol{f}_l^{-1} }{\det(\opd \boldsymbol{f}_l)},
\end{align*}
where $\opd\boldsymbol{f}_l$ denotes the Jacobian of $\boldsymbol{f}_l$, and $\iota_j, \, j=0,1,2,3$, denote the corresponding pull-backs. From \cite[Section~5.1]{da2014mathematical}, we gather that the pull-backs $\iota_0$, $\iota_1$, and $\iota_2$ are $\nabla$, $\bcurl$, and $\Div$-preserving, respectively. This yields the following discrete spaces  over an arbitrary single patch $\kappa_l$
\begin{align*}
\mathbb{S}^0_{\bp}(\kappa_l) &= \{v : v = \iota_0^{-1}(\boldsymbol{f}_l)(u), u \in \mathbb{S}^0_{\bp}(\bXi) \},\\
\mathbb{S}^1_{\bp}(\kappa_l) &= \{\bv : \bv = \iota_1^{-1}(\boldsymbol{f}_l)(\bu), \bu \in \mathbb{S}^1_{\bp}(\bXi) \},\\
\mathbb{S}^2_{\bp}(\kappa_l) &= \{\bv : \bv = \iota_2^{-1}(\boldsymbol{f}_l)(\bu), \bu \in \mathbb{S}^2_{\bp}(\bXi) \},\\
\mathbb{S}^3_{\bp}(\kappa_l) &= \{v : v = \iota_3^{-1}(\boldsymbol{f}_l)(u), u \in \mathbb{S}^3_{\bp}(\bXi) \}.
\end{align*} 
We note here that due to the properties of the pull-backs, the obtained diagram is commutative, see \cite[Section~2.2]{hiptmair2002}. 
Since we are excluding non-watertight geometries and geometries with T-junctions from our consideration, we require the parametrizations to coincide at all interfaces, i.e., for all $l_1\neq l_2,$ if $\partial K = \partial\kappa_{l_1} \cap \partial\kappa_{l_2} \neq \emptyset$, we enforce $\boldsymbol{f}_{l_1} = \boldsymbol{f}_{l_2}$. Hence, $\bp$ and $\bXi$ coincide on every patch interface, which either consists of a common surface or edge, or reduces to a single common point. Under these considerations, we call the obtained representation a multipatch domain, and define the following global B-Spline spaces in a physical domain $K = \Omega \subset \R^3$ by
\begin{equation}\label{eq:BsplinePhys}
 \begin{aligned}
\mathbb{S}^0_{\bp}(\Omega) &= \{v \in H^1(\Omega) : v_{|\Omega_l} \in \mathbb{S}^0_{\bp}(\bXi),\, \forall \,0 \leq l<N_\kappa \},\\
\mathbb{S}^1_{\bp}(\Omega) &= \{\bv \in \hcurl : \bv _{|\Omega_l} \in \mathbb{S}^1_{\bp}(\bXi),\, \forall \,0 \leq l<N_\kappa \},\\
\mathbb{S}^2_{\bp}(\Omega) &= \{\bv \in \hdiv : \bv _{|\Omega_l} \in \mathbb{S}^2_{\bp}(\bXi),\, \forall \,0 \leq l<N_\kappa \},\\
\mathbb{S}^3_{\bp}(\Omega) &= \{v \in L^2(\Omega) : v _{|\Omega_l} \in \mathbb{S}^3_{\bp}(\bXi),\, \forall \,0 \leq l<N_\kappa \},
\end{aligned}
\end{equation} 
where $ \Omega_l$ denotes a patch, and $N_\kappa$ the number of patches. With this definition, the exact sequence property still holds. \\    For $K = \partial \Omega:= \Gamma$, the global B-Spline spaces can be obtained analogously by squeezing the third dimension in \eqref{eq:BsplineParam}, i.e., by removing $\Xi_3$ and the corresponding dimension $p_3$, together with the third terms in $\mathbb{S}^1_{\bp}({\bXi})$ and $\mathbb{S}^2_{\bp}({\bXi})$. We denote by $\bpt$ and $\bXit$ the reduced vector of degrees and the reduced knot vector, respectively. We refer to \cite{buffa2020multipatch} for more details. Alternatively, they can be obtained by applying the appropriate trace operators introduced in the previous section to \eqref{eq:BsplinePhys}, namely, $\gamma_0: H^1(\Omega) \rightarrow \spht$, $\dtrace: \hcurl \rightarrow \hcurlg$, $\twtrace: \hcurl  \rightarrow \hdivg$, and $\gamma_\bn:\hdiv \rightarrow \dht$. In either cases, we obtain equivalent spaces. 
By using the latter approach, we introduce
%Using the first variant as in \cite{buffa2020multipatch}, we obtain 
% \begin{equation}\label{eq:BsplinePhysTrace}
% \begin{aligned}
%\mathbb{S}^0_{\bpt}(\Gamma) &= \{\psi \in \spht : \psi_{|\Gamma_l} \in \mathbb{S}^0_{\bpt}(\bXit),\, \forall \,0 \leq l<N_\kappa \},\\
%\mathbb{S}^{1,\parallel}_{\bpt}(\Gamma) &= \{\bpsi \in \hcurlg : \bpsi _{|\Gamma_l} \in \mathbb{S}^{1,\parallel}_{\bpt}(\bXit),\, \forall \,0 \leq l<N_\kappa \},\\
%\mathbb{S}^{1,\perp}_{\bpt}(\Gamma) &= \{\bpsi \in \hdivg : \bpsi _{|\Gamma_l} \in \mathbb{S}^{1,\perp}_{\bpt}(\bXit),\, \forall \,0 \leq l<N_\kappa \},\\
%\mathbb{S}^2_{\bpt}(\Gamma) &= \{\psi \in \dht : \psi _{|\Gamma_l} \in \mathbb{S}^2_{\bpt}(\bXit),\, \forall \,0 \leq l<N_\kappa \},
%\end{aligned}
%\end{equation} 
%where $\mathbb{S}^{1,\parallel}_{\bp}(\Gamma)$ and $\mathbb{S}^{1,\perp}_{\bp}(\Gamma)$ are both trace spaces of $\mathbb{S}^1_{\bp}(\Omega)$. This can be seen from the fact that the trace operators $\twtrace$ and $\dtrace$ act both on vector fields in $\hcurl$.
\begin{equation}\label{eq:BsplinePhysTrace}
\begin{aligned}
\mathbb{S}^0_{\bpt}(\Gamma) &= \gamma_0(\mathbb{S}^0_{\bp}(\Omega)),\\
\mathbb{S}^{1,\parallel}_{\bpt}(\Gamma) &= \dtrace(\mathbb{S}^1_{\bp}(\Omega) ),\\
\mathbb{S}^{1,\perp}_{\bpt}(\Gamma) &= \twtrace(\mathbb{S}^1_{\bp}(\Omega) ),\\
\mathbb{S}^2_{\bpt}(\Gamma) &= \gamma_\bn(\mathbb{S}^2_{\bp}(\Omega)).
\end{aligned}
\end{equation} 
Moreover, by observing that $\twtrace = \dtrace \times \bn$, we arrive at a full discretization of the de~Rham complex of Figure~\ref{fig:deRham} using conforming B-Spline spaces in Figure~\ref{fig:deRhamBsplines}, see \cite{buffa2020multipatch}. \\
%\begin{figure}[h!]
%% \begin{tikzpicture}
%\begin{tikzcd}
%& \mathbb{S}^{1,\parallel}_{\bp}(\Gamma) \ar[dr,"\curl_\Gamma"] \ar[dd,"n\times"]
%&
%&[1.5em] \\
%\mathbb{S}^0_{\bp}(\Gamma) \ar[ur, "\nabla_\Gamma"] \ar[dr, "\bcurl_\Gamma"']
%&
%& \mathbb{S}^2_{\bp}(\Gamma) \\
%& \mathbb{S}^{1,\perp}_{\bp}(\Gamma) \ar[ur,"\Div_\Gamma"']
%&
%&
%\end{tikzcd} %\end{tikzpicture}
%%\caption{test}
% \end{figure}
%
%With this, a full discretization of the discrete de~Rham complex in a three-dimensional domain $\Omega$ and its boundary $\Gamma$ can be summerized in the following diagram 
 \begin{figure}[t!]
\begin{tikzcd}
\mathbb{S}^0_{\bp}(\Omega) \arrow[r, "\nabla"] \ar[dd,"\gamma_0"]
& \mathbb{S}^1_{\bp}(\Omega) \arrow[r, "\bcurl"] \arrow[d,"\dtrace"] \ar[ddd, "\twtrace"pos=0.75, bend right=40]
& \mathbb{S}^2_{\bp}(\Omega) \arrow[r, "\Div"] \ar[dd,"\gamma_\bn"]
& \mathbb{S}^3_{\bp}(\Omega)\\
& \mathbb{S}^{1,\parallel}_{\bpt}(\Gamma) \ar[dr,"\curl_\Gamma"] \ar[dd,"\times\bn"]
&
&[1.5em] \\
\mathbb{S}^0_{\bpt}(\Gamma) \ar[ur, "\nabla_\Gamma"] \ar[dr, "\bcurl_\Gamma"']
&
& \mathbb{S}^2_{\bpt}(\Gamma) \\
& \mathbb{S}^{1,\perp}_{\bpt}(\Gamma) \ar[ur,"\Div_\Gamma"']
&
&
\end{tikzcd} 
\caption{Conforming discretization of the de~Rham complex (Figure \ref{fig:deRham}) using B-Spline spaces.} \label{fig:deRhamBsplines}
 \end{figure}
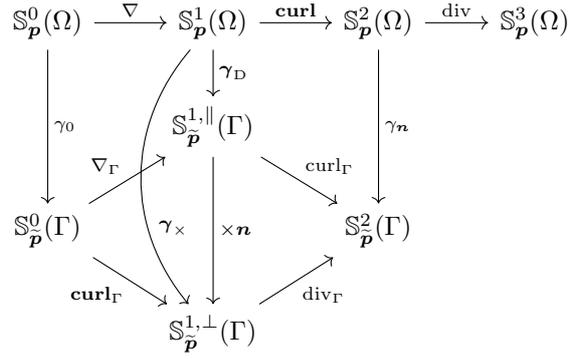%

Now we introduce a discrete counterpart for $\hdivzg$. Let $\mathbb{S}^{1,\perp}_{\bpt,0}(\Gamma)$ denote the subspace of $\mathbb{S}^{1,\perp}_{\bpt}(\Gamma)$ that contains solenoidal surface divergence conforming B-Splines, i.e.,
\begin{equation*}
\mathbb{S}^{1,\perp}_{\bpt,0}(\Gamma) = \{\bpsi \in \mathbb{S}^{1,\perp}_{\bpt}(\Gamma) : \Div_\Gamma \bpsi = 0\}.
\end{equation*}
Analogously to \cite{hiptmairSymmetric}, this subspace can be further characterized by exploiting the exactness of the de~Rham complex for trivial topologies. Hence, we obtain $\mathbb{S}^{1,\perp}_{\bpt,0}(\Gamma) = \bcurl_\Gamma \mathbb{S}^0_{\bpt}(\Gamma) $.
 \begin{remark}
 	A generalization of $\mathbb{S}^{1,\perp}_{\bpt,0}(\Gamma)$ to non-trivial topologies is possible, see \cite{hiptmair2005coupled}. It requires an additional term representing surface cohomology vector fields, namely,
 	\begin{equation*}
 		\mathbb{S}^{1,\perp}_{\bpt,0}(\Gamma) = \bcurl_\Gamma \mathbb{S}^0_{\bpt}(\Gamma) + \mathfrak{H}_\ell, \quad \dim \mathfrak{H}_\ell = \beta_1(\Gamma)
 	\end{equation*}
 	with $\beta_1(\Gamma)$ denoting the first Betti number of $\Gamma$. We refer the reader to \cite{hiptmairSymmetric} and \cite{hiptmair2002generators} for more information, and for a construction of the needed bases for $\mathfrak{H}_\ell$, respectively.
 \end{remark}
One of the main motivations in considering Isogeometric Analysis is the combination of the design process with the numerical analysis, which means again that simulations are performed directly on the designed geometric model. This is clearly more advantageous than classical approaches, since  it allows to avoid approximation errors with respect to the geometry representation. However, in some cases, such as geometries involving conic  sections (except the parabola), B-Splines are not sufficient to yield an exact representation. In such cases, NURBS can be used for geometric modeling instead. A NURBS function can be defined as follows 
 \begin{equation*}
r_\bi^\bp(\bx) := \frac{ w_\bi b_{\bi}^{\bp} (\bx)}{\sum_{\bj \in \I} w_\bj b_{\bj}^{\bp} (\bx)} ,
\end{equation*}
where $\bi \in \mathcal{I}$ is a set of multi-indices, compare with \eqref{eq:parametrization}. Replacing $b_\bi^\bp(\bx)$ in \eqref{eq:parametrization} by $r_\bi^\bp(\bx)$ leads to a NURBS parametrization, which is capable to represent free-form and sculptured surfaces as well as conic sections \cite{da2014mathematical,piegl2012nurbs}. However, due to the weighting function in the denominator $\sum_{\bj \in \I} w_\bj b_{\bj}^{\bp} (\bx)$, multivariate NURBS spaces cannot be constructed by means of tensor products. This explains the motivation behind resorting to B-Spline spaces for the discretization of the de~Rham complex. \\   
Using B-Splines in numerics and NURBS for design clearly violates one of main claims of Isogeometric Analysis, which consists in using the same basis functions in both design and analysis processes. Nevertheless, by observing the definition of a NURBS function, and since B-Splines form a partition of unity \cite{piegl2012nurbs}, we see that they can be interpreted as NURBS, which are defined on the hyperplane corresponding to $(w_i )_{i=0,\dots,k-1}= 1$. Therefore, we choose to keep labeling the presented coupling isogeometric. 

\subsection{Discrete problem and a~priori error estimates}\label{subsec:estimates}
%%%%%%%
The essence of numerical methods is to approximate infinite-dimensional spaces with suitable finite-dimensional counterparts. To this aim, we opt for a conform discretization, such that the results that we established in the continuous setting can be transferred to the discrete one.\\ Let $\bV \subset \hcurl$, and $\bX \subset \hdivg$ be some finite-dimensional subspaces with refinement level $\ell$. A sequence of refined discrete spaces has to be unterstood in the following sense
\begin{equation*}
\lim_{\ell \rightarrow \infty} \bV = \hcurl, \quad \lim_{\ell \rightarrow \infty} \bX = \hdivg.
\end{equation*}
Equivalently to the continuous setting, we may define the quotient space $[\bV] := \bV/\bVz$ with $\bVz := \{\bvl \in \bV: \bcurl \bvl = \boldsymbol{0} \}$. Furthermore, there exists an orthogonal projection $\mathsf{P}_\ell: \bV \rightarrow \bVzp $, where $\bVzp$ denotes the orthogonal complement of $\bVz$ in $\bV$. Hence, it holds that $\bVzp \cong [\bV]$. For the Neumann data, we define the discrete counterpart of $\hdivzg$ by $\bXz := \{ \bpsil \in \bX: \Div \bpsil = 0 \}$. \\
Replacing $\bH:=[\hcurl] \times \hdivzg$ by $\bHl:= [\bV] \times \bXz$ in the variational problem \ref{problem:varProb} yields a discrete weak formulation.
\begin{problem}\label{problem:varProbdisc}
 Find $\uvecl: = (\bul,\bphil) \in \bHl $ such that 
\begin{align*}
 \scalar{\M\bcurl\bul,\bcurl \bvl }- \duality{\bphil, \dtrace \bvl}_\Gamma &= \scalar{\boldsymbol{f} ,\bvl} + \duality{\boldsymbol{\phi}_0, \dtrace \bvl}_\Gamma, \\ 
 \duality{\bpsil, \SL \bphil}_\Gamma + \duality{\bpsil , (\frac{1}{2} + \DL) \dtrace \bul}_\Gamma &= \duality{\bpsil , (\frac{1}{2} + \DL) \boldsymbol{u}_0}_\Gamma 
\end{align*}  
for all $\vvecl:= (\bvl,\bpsil) \in \bHl$.
\end{problem}
Due to the conforming discretization, Theorem~\ref{theorem:wellposedness} applies. Consequently, there exists a unique $\uvecl: = (\bul,\bphil) \in \bHl:= [\bV] \times \bXz$ that solves Problem~\ref{problem:varProbdisc} under the same conditions, namely, Lipschitz continuity and strong monotonicity of $\M$ provided $C_\mathrm{M}^\M > \frac{1}{4} C_\DL$. \\
Analogously to \cite[Theorem~3.2]{EEK20}, there holds quasi-optimality in the sense of the C\'ea-type Lemma, i.e., for $\uvec = (\bu,\bphi) \in \bH $ being the solution of Problem~\ref{problem:varProb} and $\uvecl: = (\bul,\bphil) \in \bHl$ the solution of Problem~\ref{problem:varProbdisc}, there exists $C_\mathrm{C} = \frac{C_L}{C_M}$ such that  
\begin{equation}\label{eq:quasioptimality}
\begin{split}
\norm{\bu-\bul}_{\hcurl} + \norm{\bphi- \bphil}_\hdivg \leq C_\mathrm{C} \min_{\bvl \in \bV, \bpsil \in \bXz} &(  \norm{\bu-\bvl}_{\hcurl} \\ &+ \norm{\bphil- \bpsil}_\hdivg ).
\end{split} 
\end{equation}

Henceforth, we set $\bV = \mathbb{S}^1_{\bp}(\Omega)$ and $\bX = \mathbb{S}^{1,\perp}_{\bpt}(\Gamma)$. 
%In order to fix the representatives in the B-Spline space $\mathbb{S}^1_{\bp}(\Omega)$, we consider the orthogonal projection $\mathsf{P}_\ell \mathbb{S}^1_{\bp}(\Omega)$, such that the approximation properties of $\mathbb{S}^1_{\bp}(\Omega)$ can be transferred to the chosen representatives, see the proof of Theorem \ref{theorem:PrioriEstimate} for more details. 
In the previous section, we introduced the space 
$\mathbb{S}^{1,\perp}_{\bpt,0}(\Gamma) = \bcurl_\Gamma \mathbb{S}^0_{\bpt}(\Gamma) $ as a suitable discrete space for the Neumann data. Therefore, we set $\bXz = \mathbb{S}^{1,\perp}_{\bpt,0}(\Gamma)$.
\begin{definition}\label{def:patchwiseReg}
Let $K = \{\Omega, \Gamma\}$ be a multipatch domain with $N_\kappa$ patches, and let $\boldsymbol{\operatorname{d}}$ be some differential operator. For some $s \in \R$, we define the space of patchwise regularity endowed with the norm
\begin{align*}
 \label{eq:pwnorm}
\norm{\bu}_{\boldsymbol{H}_\mathrm{pw}^{s}(\boldsymbol{\operatorname{d}},K)}^2 = \sum_{0<i\leq N_\kappa} \norm{\bu_{|K_i}}_{\boldsymbol{H}^{s}(\boldsymbol{\operatorname{d}},K)}^2
\end{align*}
by 
\begin{equation*}
\boldsymbol{H}_\mathrm{pw}^{s}(\boldsymbol{\operatorname{d}},K) = \{ \bu \in \boldsymbol{L}^2(K): \norm{\bu}_{\boldsymbol{H}_\mathrm{pw}^{s}(\boldsymbol{\operatorname{d}},K)} < \infty \}.
\end{equation*}
Thereby, $\boldsymbol{H}^{s}(\boldsymbol{\operatorname{d}},K)$ denotes an energy space of regularity $s$, formally,  
\[\boldsymbol{H}^{s}(\boldsymbol{\operatorname{d}},K) := \{ \bv \in \boldsymbol{H}^{s}(K): \, \boldsymbol{\operatorname{d}}\bv \in \boldsymbol{H}^{s}(K) \}.\]
\end{definition}
The approximation properties of B-Spline spaces for the discretization of the de~Rham complex in Figure \ref{fig:deRham} are provided in \cite{da2014mathematical} and \cite{buffa2020multipatch} for multipatch domains and their trace spaces. In particular, for $\mathbb{S}^1_{\bp}(\Omega)$ and $\mathbb{S}^{1,\perp}_{\bpt}(\Gamma)$, we state the following results. 
\begin{lemma}\label{lemma:approxProp}
Let $\bu\in \hcurl \cap \boldsymbol{H}_\mathrm{pw}^{s}(\bcurl,\Omega)$ and $\bphi\in \hdivg\cap \boldsymbol{H}_\mathrm{pw}^{s}(\Div_\Gamma,\Gamma)$. There exists $C_0,\,C_{1,1},\,C_{1,2} >0$ such that
\begin{subequations}
\begin{align}
\inf_{\bul \in \mathbb{S}_{\bp}^1(\Omega) }\norm{\bu-\bul}_{\hcurl}  &\leq C_0\, h^{s} \norm{\bu}_{\boldsymbol{H}_\mathrm{pw}^{s}(\bcurl,\Omega)},  && 2< s \leq p, \label{eq:approxHcurl}\\
\inf_{\bphil \in\mathbb{S}^{1,\perp}_{\bpt}(\Gamma)}\norm{\bphi-\bphil}_\hdivg &\leq C_{1,1}\, h^{s} \norm{\bphi}_{\boldsymbol{H}^{s-\frac{1}{2}}(\Div_\Gamma,\Gamma)}, && 0\leq s \leq \frac{1}{2}, \\
\inf_{\bphil \in\mathbb{S}^{1,\perp}_{\bpt}(\Gamma)}\norm{\bphi-\bphil}_\hdivg &\leq C_{1,2}\, h^{s} \norm{\bphi}_{\boldsymbol{H}_\mathrm{pw}^{s-\frac{1}{2}}(\Div_\Gamma,\Gamma)}, && \frac{1}{2}\leq s \leq p+\frac{1}{2}. 
\end{align}
\end{subequations}
\end{lemma}
\begin{proof}
The first estimate in $\hcurl$ is given in \cite[Corollary~5]{buffa2020multipatch}. The $\hdivg$ estimates are proved in \cite[Theorem~3]{buffa2020multipatch}.
\end{proof}
\begin{theorem}[A~priori estimate] \label{theorem:PrioriEstimate}
Let $C^\M_\mathrm{E}> \frac{1}{4}C_\DL$. Moreover, let $(\bu,\bphi) \in \bH$ be the solution of the Problem \ref{problem:varProb} and 
let $(\bul,\bphil) \in \bHl = \mathbb{S}_{\bp}^1(\Omega) \times \mathbb{S}^{1,\perp}_{\bpt,0}(\Gamma)$ 
be the solution of the discrete Problem \ref{problem:varProbdisc}.
Then, there exists a constant $C>0$ such that:
\begin{itemize}
\item For $ 0 \leq s \leq \frac{1}{2}$, 
\begin{equation*}
\norm{\bu-\bul}_{\hcurl} + \norm{\bphi-\bphil}_\hdivg \leq C h^{s}\left( \norm{\bu}_{\boldsymbol{H}_\mathrm{pw}^{2}(\bcurl,\Omega)} + \norm{\bphi}_{\boldsymbol{H}^{s-\frac{1}{2}}(\Div_\Gamma,\Gamma)}\right)
\end{equation*}
with $\bu\in \hcurl \cap \boldsymbol{H}_\mathrm{pw}^{2}(\bcurl,\Omega)$ and $\bphi\in \hdivg\cap \boldsymbol{H}^{s-\frac{1}{2}}(\Div_\Gamma,\Gamma)$.
\item For $\frac{1}{2} \leq s \leq 2$,
\begin{equation*}
\norm{\bu-\bul}_{\hcurl} + \norm{\bphi-\bphil}_\hdivg \leq C h^{s}\left( \norm{\bu}_{\boldsymbol{H}_\mathrm{pw}^{2}(\bcurl,\Omega)} + \norm{\bphi}_{\boldsymbol{H}_\mathrm{pw}^{s-\frac{1}{2}}(\Div_\Gamma,\Gamma)}\right)
\end{equation*}
with $\bu\in \hcurl \cap \boldsymbol{H}_\mathrm{pw}^{2}(\bcurl,\Omega)$ and $\bphi\in \hdivg\cap \boldsymbol{H}_\mathrm{pw}^{s-\frac{1}{2}}(\Div_\Gamma,\Gamma)$.
\item For $  2 < s \leq p$,
\begin{equation*}
\norm{\bu-\bul}_{\hcurl} + \norm{\bphi-\bphil}_\hdivg \leq C h^{s}\left( \norm{\bu}_{\boldsymbol{H}_\mathrm{pw}^{s}(\bcurl,\Omega)} + \norm{\bphi}_{\boldsymbol{H}_\mathrm{pw}^{s-\frac{1}{2}}(\Div_\Gamma,\Gamma)}\right)
\end{equation*}
with $\bu\in \hcurl \cap \boldsymbol{H}_\mathrm{pw}^{s}(\bcurl,\Omega)$ and $\bphi\in \hdivg\cap \boldsymbol{H}_\mathrm{pw}^{s-\frac{1}{2}}(\Div_\Gamma,\Gamma)$.
\end{itemize}
\end{theorem}
\begin{proof}
Due to Theorem \ref{theorem:wellposedness}, Problem \ref{problem:varProb} possesses a unique solution in $[\hcurl]\times \hdivzg$. Moreover, a conform Galerkin discretization means that well-posedness of the discrete Problem \ref{problem:varProbdisc} follows from the well-posedness of the continuous setting. Indeed, let $\bu \in \hcurl$. We know from \cite[Appendix~A]{buffa2020multipatch} that there exists a projection $\Pi: \hcurl\rightarrow \mathbb{S}_{\bp}^1(\Omega)$, which commutes with the respective differential operator, namely, $\Pi \circ \bcurl = \bcurl\circ\,\Pi$. Let $ \mathsf{P}: \hcurl \rightarrow (\nabla H^1(\Omega))^\perp \subset \hcurl$ be an orthogonal projection, which is defined analogously to the one used in the proof of Proposition \ref{proposition:representative}. Then, $\Pi_1 :=\Pi\circ \mathsf{P}: \hcurl \rightarrow [\mathbb{S}_{\bp}^1(\Omega)] \cong (\nabla \mathbb{S}_{\bp}^0(\Omega))^\perp \subset \mathbb{S}_{\bp}^1(\Omega)$ defines a projection, which retains the same convergence rates w.r.t. $h$-refinement as the ones given by $\Pi$, due to the optimality of orthogonal projections. \\ Furthermore, for smooth enough functions, the commutativity of the de~Rham complex leads to
\begin{equation*}
\gamma_\bn  \circ \bcurl =  \Div_\Gamma \circ \twtrace.
\end{equation*}
Hence, by using $\ntrace = \twtrace \bcurl$, we obtain for $\bul = \Pi\bcurl\bv = \bcurl \bvl$ with $\bvl := \Pi \bv$ \[\Div_\Gamma \ntrace \bvl = \gamma_\bn \bcurl \bcurl\bvl = 0.\] Therefore, $\bpsil := \ntrace \bvl \in \mathbb{S}^{1,\perp}_{\bpt}(\Gamma)$ satisfies $\Div_\Gamma \bpsil = 0$, i.e., $\bpsil \in \mathbb{S}^{1,\perp}_{\bpt,0}(\Gamma)$.
With this, the approximation results of Lemma \ref{lemma:approxProp} and \eqref{eq:quasioptimality} may be utilized, which lead merely to the asserted a~priori estimates.
\end{proof}
In the following, we give analoga to \cite[Theorem~6]{EEK20} and \cite[Remark~6]{EEK20}, where the possible super-convergence of the solution of the scalar problem in the BEM domain was discussed. This behavior represents one of the major benefits of using BEM, in particular if the exterior solution or some derived quantities, such as force or torque, are sought. For notational simplicity, let 
\[\bH_\mathrm{pw}^s:= (\hcurl \cap \boldsymbol{H}_\mathrm{pw}^{s}(\bcurl,\Omega))\times ( \hdivg\cap \boldsymbol{H}_\mathrm{pw}^{s}(\Div_\Gamma,\Gamma))\]
be a product space of patchwise regularity. Its corresponding norm, denoted by $\norm{\cdot}_{\bH^s_\mathrm{pw}}$, is defined accordingly to Definition \ref{def:patchwiseReg}.  
\begin{theorem}\label{aubinNitsche}
Let $\mathfrak{F}\in\bHl'$ be a continuous and linear functional. We denote by $\uvec:=(\bu,\bphi) \in \bH:= [\hcurl] \times \hdivzg $ the solution of the continuous problem \ref{problem:varProb}, and by $\uvecl:=(\bul,\bphil) \in \bHl := \mathbb{S}_{\bp}^1(\Omega) \times \mathbb{S}^{1,\perp}_{\bpt,0}(\Gamma)$ the solution of the discrete problem \ref{problem:varProbdisc}. Moreover, let $\mathbf{w} \in \bH$ be the unique solution of the dual problem
\begin{equation} \label{dualP}
\ba(\mathbf{v},\mathbf{w}) = \mathfrak{F}(\mathbf{v})
\end{equation} 
for all $\mathbf{v} \in \bH$. 
Then, there exists a constant $C_1>0$ such that
\begin{align}\label{functional1}
\vert \mathfrak{F}(\mathbf{u}) - \mathfrak{F}(\mathbf{u}_\ell) \vert &
\leq C_1 \,   \Vert \mathbf{u} - \mathbf{v}_\ell \Vert_{\bH} \Vert \mathbf{w} - \mathbf{z}_\ell \Vert_{\bH}
\end{align}
for arbitrary $\mathbf{v}_\ell\in \bHl$, $\mathbf{z}_\ell\in\bHl$.
Furthermore, provided $\uvec \in \bH_\mathrm{pw}^s$ and $\wvec \in \bH_\mathrm{pw}^t$ with $2 < s,t \leq p$, there exists a constant $C_2>0$ such that 
\begin{align}
\label{functional2}
\vert \mathfrak{F}(\mathbf{u}) - \mathfrak{F}(\mathbf{u}_\ell) \vert &
\leq C_2 \,  h^{s+t} \Vert \mathbf{u}\Vert_{\bH_\mathrm{pw}^s} \Vert \mathbf{w} \Vert_{\bH_\mathrm{pw}^t}.
\end{align}
\end{theorem}
\begin{proof}
The assertions can be shown analogously to \cite[Theorem~6]{EEK20}, where the proof is based on an Aubin-Nitsche argument similar to \cite[Theorem~4.2.14]{sauter} .
\end{proof}

\begin{remark}\label{rem:dependence_pos}
In particular, the functional of Theorem \ref{aubinNitsche} may be, e.g., the representation formula \eqref{eq:representationFormula} of the exterior domain $(\alpha = 1)$.
Assuming a maximal regularity for $\uvec$ and $\wvec$ in \eqref{functional2}, i.e., $\uvec\in\bH_\mathrm{pw}^p$ and $\wvec\in\bH_\mathrm{pw}^p$ yields the following pointwise error: For $\bx \in \Omega^\mathrm{e}$, it holds 
\begin{align}
\label{eq:superconvergence}
|\bu^\mathrm{e}(\bx) - \bul^\mathrm{e}(\bx)|=|\mathfrak{F}(\uvec^\mathrm{e})(\bx)-\mathfrak{F}(\mathbf{u}_\ell^\mathrm{e})(\bx)|\leq C h^{2p}, \quad C>0.
\end{align}
This suggests that under some regularity assumptions, the converge rate may double. This behavior is called super-convergence. Analogously to \cite[Remark~6]{EEK20}, the constant $C$ in \eqref{eq:superconvergence} depends in particular on the fundamental solution $G(\bx,\by)$. Hence, it tends to infinity, when approaching the boundary. However, suitable numerical schemes similar to \cite{schwab1999extraction} may be considered to reduce this non-desirable effect, thus, to profit from the faster convergence. 
\end{remark}
\section{Numerical illustration }\label{sec:results}
In this section, we confirm the optimal convergence rates of the isogeometric non-symmetric FEM-BEM coupling as expected from Theorem~\ref{aubinNitsche} and Remark~\ref{rem:dependence_pos}. The implementation of Boundary Integral Operators (BIOs) is performed by utilizing \textsf{BEMBEL}, which is an efficient and freely-available C++ library that provides the necessary building-blocks for the assembling of BIOs in the isogeometric framework \cite{dolz2020bembel}. In particular, we need the vectorial single-layer and double-layer BIOs, which arise from $\duality{\bpsil, \SL \bphil}_\Gamma$ and $\duality{\bpsil , \DL \dtrace \bul}_\Gamma$, respectively, and a boundary mass matrix that is given by $\duality{\bpsil , \dtrace \bul}_\Gamma$. The rest of the implementation, which includes the stiffness matrix of the isogeometric FEM-part, the right-hand sides, the coupling procedure, as well as solving and post-processing are performed in \textsf{MATLAB/Octave}. We use thereby some procedures of \textsf{GeoPDEs}, see \cite{geopdes,VAZQUEZ2016523}, in particular for the assembling of the FEM-matrix, and the construction of multipatch geometries, which is made with the included \textsf{NURBS toolbox}. \\
 
In this work, we only consider uniform $h$-refinements. That is, in every refinement level $\ell \geq 0$, the element size $h$, which corresponds to the length of a univariate element in the parametric domain, is equal to $2^{-\ell}$. Note that the obtained system of linear equations is singular due to the infinite-dimensional kernel of the $\bcurl$-operator in the FEM-domain. However, by considering a consistent right-hand side, i.e., divergence-free, and by applying suitable preconditioning techniques, an iterative CG-based solver converges properly. To ensure consistency on the discrete level, a correction of the data as proposed in \cite{kanayama2001numerical}, for instance, may be necessary. Furthermore, we refer the reader to, e.g., \cite{hiptmair2007nodal} and \cite{steinbach1998construction} for efficient FEM and BEM preconditioning, respectively.\\ In the following, we perform the numerical analysis on an academic example consisting in a uniformly magnetized unit ball with the magnetization $\boldsymbol{m} = (0,0,1)^\top$. We denote the unit ball by $B_3(\boldsymbol{0}\,;1)$, namely, a three-dimensional ball with radius $r = 1$ centered at the origin of the coordinate system.
%	We choose $\boldsymbol{f} = \bcurl \boldsymbol{m}$. This reduces the contribution of the magnetization to the surface, since the volume one obviously vanishes. 
In addition, we consider the simplest material form, which is air. Hence, $\M = \Id$, which implies that the vaccum permeability $\mu_0$ has been normalized, i.e., $\mu_0 = 1$. 
An analytical solution for this problem is given in \cite[Chapter~5.10]{jackson2009classical}. Let $(\rho,\varphi,\theta)$ and $(u_\rho, u_\varphi,u_\theta)$ denote the coordinates and the components of a vector field $\bu$ in a standard spherical system, respectively. For this setting, only the azimuthal (third) component of the solution is non-zero. It reads for $\rho \leq 1$ 
\begin{equation}\label{eq:exactSolInt}
	u_\theta = \frac{\rho}{3}  \sin \theta,
\end{equation}
and for $\rho > 1$
\begin{equation}\label{eq:exactSolExt}
 u^\mathrm{e}_\theta = \frac{1}{3\rho^2}  \sin \theta.
\end{equation}
To solve the non-symmetric FEM-BEM coupling, we need to specify the corresponding right-hand sides for Problem \ref{problem:varProb}, which can be derived by using the exact solutions $\bu = (0,0,u_\theta)$ and $\bue = (0,0,u^\mathrm{e}_\theta)$. Concretely, we prescribe 
\begin{equation}
	\boldsymbol{f}= \boldsymbol{0}, \quad \boldsymbol{u}_0 = \boldsymbol{0}, \quad \bphi_0 = \bn \times \boldsymbol{m},
\end{equation}
where $\bn$ is an outer pointing normal vector. Note that $\boldsymbol{f}= \boldsymbol{0}$ means that there are no impressed currents and that the volume contribution of the magnetization vanishes. \\
\begin{figure}[!t]
	\begin{center}
		\includegraphics[width=10cm]{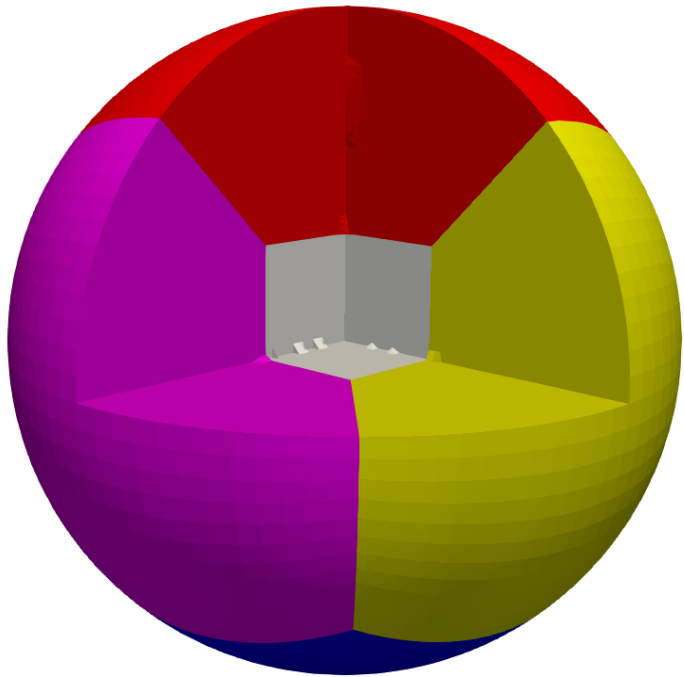}
	\end{center}
	\caption{Multipatch representation of the ball $B_3(\boldsymbol{0}\,;1)$. It consists of seven patches that are glued at interfaces. For visualization purpose, we show only the domain $B_3(\boldsymbol{0}\,;1)\backslash[0,1]^3$.} \label{fig:multipatchDomain}
\end{figure}%
For the isogeometric approach, the ball $B_3(\boldsymbol{0};1)$ has to be represented as a multipatch domain. For instance, we propose in Figure \ref{fig:multipatchDomain} a possible representation. It is composed of seven NURBS patches that are glued at interfaces, as explained in Section \ref{sec:iga}. Note that the choice of the patches is arbitrary, as long as the parametrization \eqref{eq:parametrization} is regular. \\
\begin{figure}[!tb]
	\begin{center}
		\begin{tikzpicture}
		\node at (0,0) 			{\includegraphics[width=12cm]{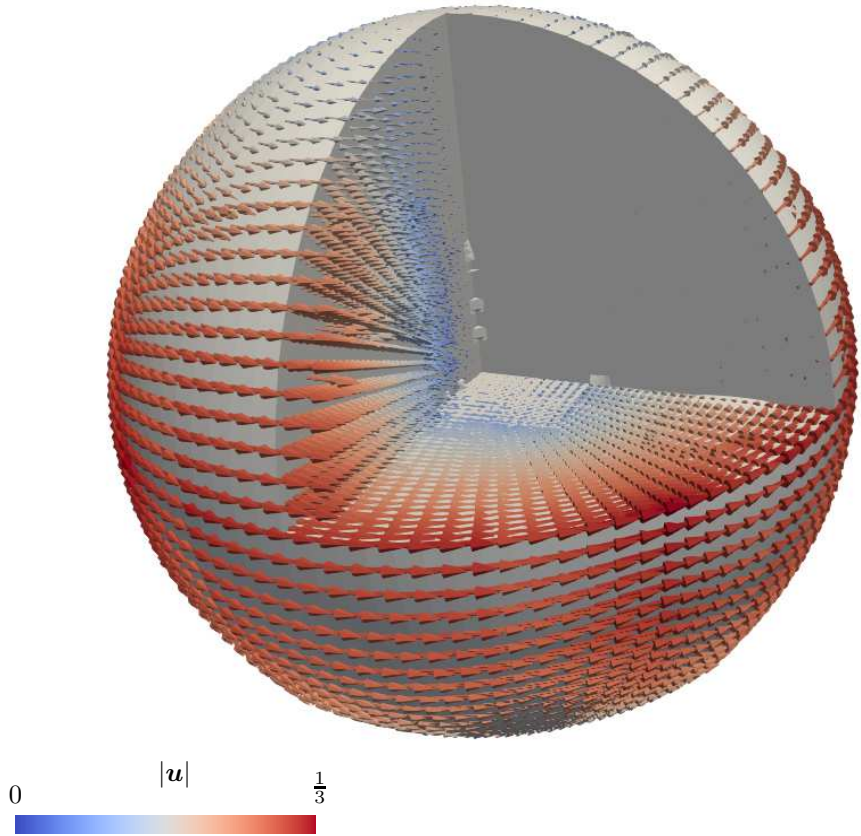}};
		\node[right] at (-4.25,-5) {$\vert\bu\vert$};
   	 	\node[above] at (-6,-5.5) {$0$};
		\node[above] at (-2,-5.5) {$\frac{1}{3}$};
		\end{tikzpicture}
	\end{center}
	\caption{Interior solution depicted in the domain $B_3(\boldsymbol{0}\,;1)\backslash[0,1]^3$. The arrows encode the solution $\bu$, and the color bar its magnitude. The considered B-Spline spaces have the degree $p = 2$, and a refinement level $\ell = 3$.}\label{fig:solution}
\end{figure}

To solve the discrete Problem~\ref{problem:varProbdisc}, we use the BiCGSTAB solver of MATLAB/Octave with a tolerance of $10^{-6}$. For convenience, an interior solution with a polynomial degree $p = 2$ at a refinement level $\ell = 3$ for the discrete ansatz space $\mathbb{S}_{\bp}^1(\Omega) \times \mathbb{S}^{1,\perp}_{\bpt,0}(\Gamma)$ is shown in Figure~\ref{fig:solution}. For visualization purpose, it is depicted only in $B_3(\boldsymbol{0}\,;1)\backslash[0,1]^3$. \\
%\begin{figure}[!tb]
%	\centering
%	\begin{tikzpicture}
%	\begin{loglogaxis}[width=10cm,
%	xlabel={\small $h^{-1}$},
%	ylabel={\small $\sqrt{\norm{\bu-\bul}^2_{\hcurl}+\norm{\bphi-\bphil}^2_\SL}$ },
%	font={\scriptsize},grid=major,
%	legend entries={$p=1$,$p=2$},
%	legend style={font=\small,at={(0.025,0.025)},
%		anchor=south west,thin}]
%	\addplot [color=black, mark=o, line width=0.5pt, mark size=2.5pt, mark options={solid, black}] table {data/convergence_p1_a__error_H_0to4.dat};
%	\addplot [color=red, mark=square, line width=0.5pt, mark size=2.5pt, mark options={solid, red}] table {data/convergence_p2_a__error_H_0to4.dat};
%	\logLogSlopeTriangleInv{0.85}{-0.125}{0.3}{-2}{black}{$-2$};
%	\logLogSlopeTriangleInv{0.85}{-0.125}{0.65}{-1}{black}{$-1$};
%	\end{loglogaxis}
%	\end{tikzpicture}
%	\caption{Convergence of discrete solution $(\bu_\ell,\bphi_\ell)\in\bHl$ to the solution $(\bu,\bphi)\in\bH$.
%		The considered B-Spline spaces have the degrees $p=1,2$, respectively.} \label{fig:convergenceInt}
%\end{figure}%
%The error in Theorem \ref{theorem:PrioriEstimate} is measured in $\norm{\cdot}_\bH$, as defined in \eqref{eq:normH}. However, due to the non-locality of $\norm{\cdot}_\hdivg$, we consider the equivalent norm $\norm{\cdot}^2_\SL = \duality{\cdot,\SL \cdot}_\Gamma$ instead. The convergence analysis for the interior solution $\uvecl = (\bul,\bphil)$ is conducted for $p = 1,2$. We see in Figure \ref{fig:convergenceInt} the predicted convergence behavior. \\
\begin{figure}[!t]
	\centering
	\begin{tikzpicture}
	\begin{loglogaxis}[width=10cm,
	xlabel={\small $h^{-1}$},
	ylabel={\small error},grid=major,
	legend entries={$ p =1 $,$ p = 2 $},
	legend style={font=\small,at={(0.025,0.025)},
		anchor=south west,thin}
	]
	\addplot [color=black, mark=o, line width=0.5pt, mark size=2.5pt, mark options={solid, black}] table {data/convergence_p1_ae__error_2_0to4.dat};
	\addplot [color=red, mark=square, line width=0.5pt, mark size=2.5pt, mark options={solid, red}] table {data/convergence_p2_ae__error_2_0to4.dat};
	\logLogSlopeTriangleInv{0.85}{-0.12}{0.235}{-4}{black}{$-4$};
	\logLogSlopeTriangleInv{0.85}{-0.12}{0.55}{-2}{black}{$-2$};
	\end{loglogaxis}
	\end{tikzpicture}
	\caption{Convergence of the exterior solution. 
The $\text{error} = \max_{i=1,\ldots,N} \vert \bu^\mathrm{e}(\bx_i) - \bu^\mathrm{e}_\ell(\bx_i) \vert$ is calculated with $N=20$ evaluation points on $\partial B_3(\boldsymbol{0}\,;1.5)$. The considered B-Spline spaces have the degrees $p=1,2$, respectively.}\label{fig:convergenceExt}
\end{figure}
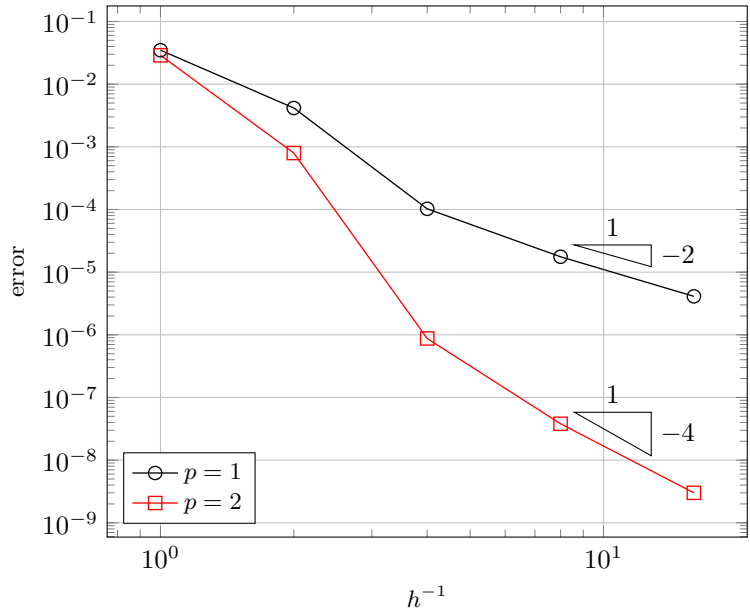%
With the solution $(\bul,\bphil)$ of Problem \ref{problem:varProbdisc} at hand, we evaluate the exterior solution at some $\bx \in \Omega^\mathrm{e}$ by using the functional
 \begin{equation*}
\mathfrak{F}(\mathbf{u}_\ell^\mathrm{e})(\bx) := \bu^\mathrm{e}_\ell(\bx) =  -\left(\bPsi_\mathrm{SL}(\bphil)(\bx) + \bPsi_\mathrm{DL}(\dtrace \bul)(\bx) \right),
\end{equation*}       
which follows from the representation formula \eqref{eq:representationFormula} by setting $\gamma^\mathrm{e}_\bn \bue = 0$ and $\bu_{0,\ell} = \boldsymbol{0}$.\\
% This is explained by the choice of consistent right-hand sides. In this case, we can verify that $\gamma^\mathrm{e}_\bn \bue = 0$. Moreover, the jump $\bu_{0,\ell} = \boldsymbol{0}$, which yields a meaningful physical behavior, since we expect the tangential component of the vector potential to be continuous across the boundary, i.e., $\dtrace \bu = \dtracex \bue $.
Now, we consider an $h$-refinement for different polynomial degrees $p = 1,2$ and verify the convergence behavior of the exterior solution. For the evaluation of the pointwise error 
\begin{equation}
\label{pointwiseError}
\text{error} = \max_{i=1,\ldots,N} \vert \bu^\mathrm{e}(\bx_i) - \bu^\mathrm{e}_\ell(\bx_i) \vert,
\end{equation}
we choose $N=20$ equally distributed points on the boundary $\partial B_3(\boldsymbol{0}\,;1.5)$. 
In Figure \ref{fig:convergenceExt}, the expected doubling of the convergence rates, as pointed out in Remark \ref{rem:dependence_pos}, is observed. \\

The current implementation is only made to prove the concept, since it is only suitable for medium sized problems. A more involved one would require fast assembling techniques for the dense BEM matrices as given in \cite{doelz2019isogeometric}, for instance, as well as efficient preconditioning for iterative solvers. In particular, it is worth mentioning that an efficient preconditioning technique, which is tailored for B-Spline spaces still poses an open problem. For these reasons, we postpone practical convergence analysis for non-linear equations to future works. Nevertheless, we expect standard techniques for non-linear equations such as a Picard-iteration scheme, as done, e.g., in \cite{EEK20} for the two-dimensional scalar case, to satisfy the theoretical estimates of Theorem \ref{theorem:PrioriEstimate} and Remark \ref{rem:dependence_pos}.  
%%%%%%%
\section{Conclusions and Outlook}\label{sec:conclusion}
We established in this work the well-posedness of the Johnson-N\'ed\'elec coupling for the magnetostatic vector potential formulation as an alternative to other approaches. Besides, our analysis allows non-linear equations in the interior domain. For this, we considered the standard framework of Lipschitz and strongly monotone operators. To counteract the infinite-dimensional kernel of the $\bcurl$-operator, we projected the equations onto suitable quotient spaces, and showed there the validity of the arising weak formulation. The proofs were made under the assumption that the domain is simply-connected. However, this can be extended to more general topologies. Furthermore, we derived a~priori estimates for a conforming discretization using B-Spline spaces that satisfy the de~Rham sequence. Thanks to their definition, B-Splines allow a straightforward $h$- and $p$- refinement, without altering the original geometry, which is represented exactly with these bases (in some cases, NURBS or other generalizations can be required instead). By means of an academic problem, we verified the predicted convergence rate in the exterior domain, which underlines one major advantage of using BEM in the context of electromagnetics, where in general not only the solution is aimed at but also some derived quantities such as the magnetic field in some points of the space, or forces and torques acting on subdomains. In future works, this method could be naturally extended both numerically and analytically to incorporate more general topologies, and multiple domains. This would provide a rigorous analysis of relevant applications in engineering sciences, such as electric motors.       
\nocite{*}
\bibliography{igafembemcurlcurl_preprint} 
\bibliographystyle{alpha}
\end{document}